\documentclass[10pt]{amsart}
\usepackage{amsmath}
\usepackage{amssymb}
\usepackage{mathtools}
\usepackage[mathscr]{eucal}
\usepackage{esint}
\usepackage{tikz}

\newcommand{\quadd}{\hspace{1.5em}}
\theoremstyle{plain}
\newtheorem{theorem}{Theorem}[section]
\newtheorem{lemma}[theorem]{Lemma}
\newtheorem{prop}[theorem]{Proposition}

\theoremstyle{definition}

\newtheorem{defn}[theorem]{Definition}

\numberwithin{equation}{section}

\begin{document}

\title[Singular Solutions of the Loewner-Nirenberg Problem]
{Singular Solutions of the Loewner-Nirenberg Problem in Conic Domains with Prescribed Singularity at Vertices}

\author[Zhou]{Stephen Zhou}
\address{Department of Mathematics\\
University of Notre Dame\\
Notre Dame, IN 46556, USA} 
\email{szhou4@nd.edu}

\begin{abstract} 
We study positive singular solutions of the Loewner-Nirenberg problem on conical domains and establish the existence of solutions that admit prescribed asymptotic expansions near vertices, valid to arbitrarily high order of approximation.
\end{abstract}

\date{\today}

\maketitle 
\section{Introduction}\label{sec-INTRO}
The study of singular solutions to conformally invariant equations has a long history, beginning with the classical Yamabe problem and continuing with the boundary blow-up problem introduced by Loewner and Nirenberg \cite{LN}. Specifically, they studied positive solutions of
\begin{align}
  \Delta u &= \tfrac14\,n(n-2)\,u^{\frac{n+2}{n-2}}
       \quad\text{in }\Omega, \label{eq:LN} \\
  u &= \infty
       \quad\text{on }\partial\Omega, \label{eq:LNB}
\end{align}
where $\Omega$ is a bounded domain in $\mathbb{R}^n$ for $n\geq3$. An improved version of the main result in \cite{LN} states that given a bounded Lipschitz domain $\Omega$ in $\mathbb{R}^n$, there exists a unique positive solution $u$ of \eqref{eq:LN}-\eqref{eq:LNB}, and furthermore, if $\partial\Omega$ is $C^{1,\alpha}$ for some $\alpha\in[0,1)$, then
\begin{equation}\label{eq:LNblowup}
|d^{\frac{n-2}{2}}u-1|\leq C d^\alpha\quad\text{near }\partial\Omega,
\end{equation}
where $d$ is the distance function on $\Omega$ to $\partial\Omega$ and $C$ is a positive constant depending on $\Omega$. In particular, \eqref{eq:LNblowup} reveals that the solution $u$ has the precise blow-up rate $d^{-\frac{n-2}{2}}$. A natural problem, addressed in later works, is to go beyond this leading-order behavior and develop full asymptotic expansions of solutions near the singular boundary. For example, if $\Omega$ has a smooth boundary, Mazzeo \cite{MAZZEO2} and Andersson, Chruściel, and Friedrich \cite{ACF} proved that solutions of \eqref{eq:LN}-\eqref{eq:LNB} admit a polyhomogeneous expansion near $\partial\Omega$.

Recently, there have been some studies of asymptotic behaviors of solutions of \eqref{eq:LN}-\eqref{eq:LNB} if $\Omega$ is only Lipschitz. 
Han and Shen \cite{HanShen2020} studied \eqref{eq:LN}-\eqref{eq:LNB} in a class of Lipschitz domains and 
identified the leading terms in the asymptotic expansions. 
Han, Jiang, and Shen \cite{HJS2024} studied the asymptotic behaviors of solutions of \eqref{eq:LN}-\eqref{eq:LNB} if $\Omega$ is a finite cone. In this case, it is natural to first consider a positive solution $u_V$ on the infinite cone $V$ over some spherical domain $\Sigma\subsetneq \mathbb{S}^{n-1}$ satisfying
\begin{align}
  \Delta u_V 
    &= \frac14n(n-2)\,u_V^{\frac{n+2}{n-2}}
    \quad \text{in }V, \label{eq:LNINFCONE}
    \\
  u_V&=\infty
    \quad \text{on }\partial V. \label{eq:LNINFCONEB}
\end{align}
According to \cite{HJS2024}, under sufficient regularity assumptions on $\Sigma$, \eqref{eq:LNINFCONE}-\eqref{eq:LNINFCONEB} admits a unique positive solution $u_V$, which in polar coordinates $x = r\theta$ takes the form
\begin{equation}\label{eq:firstxi}
u_V(x)=|x|^{-\frac{n-2}{2}}\xi (\theta),
\end{equation}
where $\xi$ is a smooth function on $\Sigma$ and $\xi=\infty$ on $\partial\Sigma$. We refer to $u_V$ as the radial solution of \eqref{eq:LN}.

The main question is how a solution $u$ of 
\begin{align}
    \Delta u &=\tfrac14\,n(n-2)\,u^{\frac{n+2}{n-2}}
    \quad \text{in }V\cap B_1, \label{eq:LNFINITECONE}
    \\
    u &= \infty
        \quad \text{on }\partial V\cap B_1,\label{eq:LNFINITECONEB}
\end{align}
compares to $u_V$. While both $u$ and $u_V$ diverge on $\partial V \cap B_1$, it is not obvious whether their ratio $u/u_V$ remains bounded near the boundary. Han, Jiang, and Shen proved that this ratio is not only bounded but in fact has a full asymptotic expansion with leading term $1$. Their main theorem can be stated as follows.
\begin{theorem}[{\cite{HJS2024}}]\label{thm:HJS}
For $n \geq 3$, let $V$ be an infinite Euclidean cone over some Lipschitz domain 
$\Sigma \subsetneq S^{n-1}$ and $u_V \in C^\infty(V)$ the unique positive solution of 
\eqref{eq:LNINFCONE}-\eqref{eq:LNINFCONEB}. Then there exist a constant $\tau$ and an increasing sequence of 
positive constants $\{\gamma_i\}_{i=1}^\infty$, with $\gamma_i \to \infty$, such that for any positive 
solution $u \in C^\infty(V \cap B_1)$ of \eqref{eq:LNFINITECONE}-\eqref{eq:LNFINITECONEB} and any integer $m \geq 0$, 
\begin{equation}\label{eq:HJS}
\left|
  (u_V^{-1}u)(x) - 1
  - \sum_{i=1}^m \;\sum_{j=0}^{i-1}
     c_{ij}(\theta)\,|x|^{\gamma_i}\,(-\ln|x|)^j
\right|
\;\leq\;
C\, d_\Sigma^\tau \, |x|^{\gamma_{m+1}}\,(-\ln|x|)^m ,
\end{equation}
where $C>0$ is a constant, $d_\Sigma$ denotes the distance function on $\Sigma$ to $\partial\Sigma$, 
and each $c_{ij}$ is a bounded smooth function on $\Sigma$ satisfying $c_{ij} = O(d_\Sigma^\tau)$.
\end{theorem}
We note that the sequence $\{\gamma_i\}_{i=1}^\infty$ in Theorem~\ref{thm:HJS}, which we call the index set and denote by $\mathcal{I}$, associated to $\Sigma$ is determined solely by the cone $V$ (or equivalently the spherical domain $\Sigma$), and is independent of the particular solution $u$ of \eqref{eq:LNFINITECONE}-\eqref{eq:LNFINITECONEB}. Moreover, the growth rate $\tau$ near $\partial \Sigma$ is also determined only by $\Sigma$, and it can be shown that $\tau > \frac{n-2}{2}$. On the other hand, the bounded smooth functions $c_{ij}(\theta)$ defined on $\Sigma$ are determined by the specific solution $u$, and can be computed in a rather mechanical way. Naturally, the expression 
\begin{equation}\label{eq:approxsoln}
u_V(x)\bigl[1+\sum_{i=1}^m \;\sum_{j=0}^{i-1}
     c_{ij}(\theta)\,|x|^{\gamma_i}\,(-\ln|x|)^j\bigr]
\end{equation}
is considered an approximate solution of \eqref{eq:LNFINITECONE}-\eqref{eq:LNFINITECONEB}. The estimate \eqref{eq:HJS} simply asserts that any positive solution $u$ of \eqref{eq:LNFINITECONE}-\eqref{eq:LNFINITECONEB} is well approximated by an approximate solution of the form \eqref{eq:approxsoln}. In this sense, \eqref{eq:HJS} associates to each actual solution 
a canonical finite asymptotic expansion.

In this paper, we study the converse question: given an approximate solution $\hat{u}$ of \eqref{eq:LNFINITECONE}-\eqref{eq:LNFINITECONEB} of a certain order, can we construct an actual solution $u$ of \eqref{eq:LNFINITECONE}-\eqref{eq:LNFINITECONEB} which is well approximated by $\hat{u}$? To answer this question, we must first define an appropriate notion of approximate solution, which is modeled after \eqref{eq:approxsoln}. In contrast to Theorem~\ref{thm:HJS}, where the coefficients 
$c_{ij}(\theta)$ are extracted from a given solution, in our construction 
the roles are reversed: we begin by prescribing a choice of asymptotic data and then construct an actual solution realizing that data. In particular, the “free data’’ 
lies precisely in the finite set of constants $c_i$ appearing in the choice of solution
\[
   \eta(t,\theta)=\sum_{i=1}^{k_1} c_i e^{-\gamma_i t}\phi_i(\theta)
\]
to the linearized problem (the functions $\phi_i$ are
eigenfunctions of the linearized operator, introduced below). Here, $\gamma_1<\cdots<\gamma_{k_1}<\gamma_{k_1+1}<\cdots$ are the 
elements of the index set $\mathcal{I}$ listed in increasing order, and 
$k_1$ is the largest index such that $\gamma_{k_1}$ is 
\emph{nonresonant}—that is, $\gamma_{k_1+1}$ is the first element of 
$\mathcal{I}$ that can be written as a nontrivial finite linear 
combination of $\gamma_1,\dots,\gamma_{k_1}$ with nonnegative integer coefficients.  
Once the constants $c_1,\dots,c_{k_1}$ are chosen, all higher-order 
coefficients $c_{ij}(\theta)$ are uniquely determined by the recursive 
structure of the nonlinear equation.  
Thus the asymptotic expansion is encoded by this finite list 
of free constants. We begin with some standard reductions. Write $\rho= \xi^{-\frac{2}{n-2}}$. As shown in \cite{HJS2024}, $\rho(\theta)$ is compatible with the distance function from $\Sigma$ to $\partial \Sigma$, i.e. $\rho(\theta)\sim d_\Sigma (\theta)$. Thus $\rho$ serves as a boundary defining function on $\Sigma$. 

To state our main result, we first introduce the operator
\begin{equation*}
    \mathcal{M}(u)= \Delta u - \tfrac14\,n(n-2)\,u^{\frac{n+2}{n-2}}.
\end{equation*}
\begin{theorem}\label{thm:MAINRESULT}
Let $\xi$ be the positive smooth function as in \eqref{eq:firstxi}, $\mathcal{I}$ the index set associated with $\Sigma$, and $\mu>\gamma_1$ with $\mu\notin\mathcal{I}$. Suppose that $\hat{u}$ is a smooth function in $V\cap B_1$ satisfying 
\begin{equation}\label{eq:assump1}
    \big||x|^{\frac{n-2}{2}}\hat{u}(x)-\xi(\theta)\big|+\rho(\theta)|x|\,\big|\nabla(|x|^{\frac{n-2}{2}}\hat{u}(x)-\xi(\theta))\big|\leq \rho(\theta)^\frac{n+2}{2} \epsilon(|x|),
\end{equation}
where $\epsilon$ is a decreasing function such that $\epsilon(|x|) \to 0$ as $|x|\to 0$, and, 
\begin{equation}\label{eq:assump2}
    |x|^{\frac{n+2}{2}}\Big(\big|\mathcal{M}(\hat{u})(x)\big|+\rho(\theta)|x|\,\big|\nabla\big(\mathcal{M}(\hat{u})\big)(x)\big|\Big)\leq C\rho(\theta)^\frac{n-2}{2}|x|^{\mu},
\end{equation}
for some positive constant $C$. Then, there exists a $R\in(0,1)$ and a positive solution $u$ of \eqref{eq:LNFINITECONE}-\eqref{eq:LNFINITECONEB} in $V\cap B_R$ such that, for any $x\in V\cap B_R$,
\begin{equation*}
    |x|^{\frac{n-2}{2}}|u(x)-\hat{u}(x)|\leq C'\rho(\theta)^{\frac{n+2}{2}}|x|^{\mu},
\end{equation*}
where $C'$ is a positive constant.
\end{theorem}

Theorem \ref{thm:MAINRESULT} states that if $\hat{u}$ is an approximate solution of the Yamabe equation which is close to the radial solution up to a prescribed order, then $\hat u$ is close to an actual solution up to a certain order. The index set $\mathcal{I}$ will be defined by \eqref{eq:INDEX} and is given by the collection $\{\gamma_i\}_{i=1}^\infty$ as in Theorem \ref{thm:HJS}. The assumptions \eqref{eq:assump1} and \eqref{eq:assump2} say that $\hat u$ is an approximate solution with leading term $|x|^{-\frac{n-2}{2}}\xi(\theta)$ of order $\mu$. Theorem \ref{thm:MAINRESULT} guarantees the existence of a positive solution of the Yamabe equation near an approximate solution with the correct order of approximation and sufficiently close to the cone's vertex.

To prove Theorem \ref{thm:MAINRESULT}, we analyze the linearized operator in suitable weighted Hölder spaces defined on a small neighborhood of the vertex of the cone and construct solutions which decay up to a prescribed order $\mu>0$. We proceed by a spectral decomposition of the linearized equation onto a finite-dimensional subspace generated by a product of eigenfunctions and exponentials, where we in turn solve finitely many ordinary differential equations. In the infinite-dimensional complementary subspace, we construct solutions via a variational method and derive a crucial estimate by a rescaling method, since the maximum principle fails due to the singular zeroth-order term. This framework follows a systematic procedure to obtain solutions of nonlinear equations starting from their linearizations: one first corrects linearized solutions to produce approximate solutions at the desired order, and then perturbs these approximate solutions to obtain genuine solutions. 

There are several methods to construct solutions of \eqref{eq:LNFINITECONE}-\eqref{eq:LNFINITECONEB}. For example, we may take an arbitrary Lipschitz domain $\Omega$ such that $\Omega\cap B_1=V\cap B_1$, find the solution of \eqref{eq:LN}-\eqref{eq:LNB}, and then restrict such a solution to $V\cap B_1$. Our construction has the advantage of using the prescribed asymptotic behaviors of solutions at the vertex of the cone as the "free" data to determine solutions. Solutions constructed in this way have the desired asymptotic behavior.

The paper is organized as follows. In Section 2, we analyze the linearized operator associated with the cone problem in suitable weighted Hölder spaces and establish key estimates. In Section 3, we prove Theorem \ref{thm:MAINRESULT} using the contraction mapping theorem. In Section 4, we present a general construction of approximate solutions via perturbations of solutions to the linearized problem. Unlike the classical isolated singularity case studied by Caffarelli, Gidas, and Spruck in \cite{CGS} and later works \cite{KMPS}, \cite{HL}, and \cite{HanLi2020}, the cone problem involves boundary singularities rather than isolated singularities. While our strategy follows the framework developed by Han and Li \cite{HanLi2020}, the geometry of the cone over 
$\Sigma$ introduces new analytical difficulties that require a different set of tools.

I would like to thank Qing Han for suggesting the problem and for his persistent encouragement.

%\newpage

\section{Linearized Yamabe Equations}\label{sec-LYE}
We begin with a standard reduction of the Loewner–Nirenberg problem via a change of variables. First, for any $x\in\mathbb{R}^n\setminus\{0\}$, using polar coordinates, let $x=r\theta$, with $r=|x|$, $\theta\in\Sigma$. Next, switching to the cylindrical variable $t=-\ln(r)$, define
\begin{equation}\label{eq:NORMALIZE}
U(t,\theta)=r^{\frac{n-2}{2}}u(r\theta).
\end{equation}
We view the equation \eqref{eq:LN} on $\mathbb{R}_+\times\Sigma$ with the cylindrical metric $dt^2+d\theta ^2$. In these coordinates, $U$ satisfies
\begin{equation}\label{eq:CYLINDRICALLN}
  \partial_{tt}U
  + \Delta_\theta U
  - \tfrac14\,(n-2)^2\,U
  = \tfrac14\,n(n-2)\,U^{\frac{n+2}{n-2}}.
\end{equation}
Suppose $u_V$ is a positive solution of \eqref{eq:LNINFCONE}-\eqref{eq:LNINFCONEB} on the infinite cone $V$ over $\Sigma$.
Define
\begin{equation}\label{eq:xidef}
\xi(\theta)=r^{\frac{n-2}{2}}u_V(r\theta).
\end{equation}
A straightforward computation shows that $\xi$ satisfies
\begin{align}
  \Delta_\theta \xi 
    - \tfrac14(n-2)^2 \,\xi 
    &= \tfrac14\,n(n-2)\,\xi^{\frac{n+2}{n-2}} \quad \text{in }\Sigma, \label{eq:xi}
    \\
  \xi &= \infty
    \quad \text{on }\partial\Sigma. \label{eq:xiB}
\end{align}
We note that the existence and uniqueness of a smooth positive solution $\xi$ of \eqref{eq:xi}-\eqref{eq:xiB} are established by the same method as in the classical Loewner–Nirenberg problem on bounded Lipschitz domains. The existence and uniqueness of a positive smooth solution $u_V$ of \eqref{eq:LNINFCONE}-\eqref{eq:LNINFCONEB} of the form \eqref{eq:xidef} follows. See \cite{HJS2024} for details.

The function $\xi$ introduced above is a key component in the analysis throughout the rest of this paper. We summarize some of its properties below. In the proceeding results, let $d_{\Sigma}$ be the distance function on $\Sigma$ to $\partial\Sigma$.
\begin{lemma}\label{lem:gradxi}
Let $\Sigma \subsetneq S^{n-1}$ be a Lipschitz domain. Then, there exists a unique positive solution $\xi \in C^\infty(\Sigma)$ of \eqref{eq:xi}-\eqref{eq:xiB}. Moreover,
\begin{equation}\label{eq:xiBLOWUP}
    c_1 \leq d_{\Sigma}^{\frac{n-2}{2}} \xi \leq c_2 \quad \text{in } \Sigma,
\end{equation}
and for any integer $k \geq 0$,
\begin{equation}\label{eq:xigradientbound}
    d_{\Sigma}^{\frac{n-2}{2}+k} |\nabla_\theta^k \xi| \leq C \quad \text{in } \Sigma,
\end{equation}
where $c_1$, $c_2$, and $C$ are positive constants depending only on $n$, $k$, and $\Sigma$. 
\end{lemma}
Estimates of the form \eqref{eq:xiBLOWUP}-\eqref{eq:xigradientbound} are standard in the discussion of the Loewner-Nirenberg problem and the classical proofs are easily modified to obtain \eqref{eq:xiBLOWUP}-\eqref{eq:xigradientbound} for solutions $\xi$ of \eqref{eq:xi}-\eqref{eq:xiB}.

In light of the blow-up rate \eqref{eq:xiBLOWUP}, we set 
\begin{equation}\label{eq:rhodef}
\rho= \xi^{-\frac{2}{n-2}}
\end{equation}
and obtain an equivalent version of \eqref{eq:xi}-\eqref{eq:xiB} given by
\begin{align}
  \rho\,\Delta_\theta\rho + S\,\rho^2 
    &= \frac{n}{2}\bigl(|\nabla_\theta\rho|^2 - 1\bigr)
    \quad \text{in }\Sigma, \label{eq:rho}
\\
  \rho &= 0 \quad \text{on }\partial\Sigma, \label{eq:rhoB}
\end{align}
where $S$ is a constant defined by
\begin{equation}\label{eq:S}
S=\frac{1}{2}(n-2).
\end{equation}
The constant $S$ introduced above is related to the scalar curvature of $\mathbb{S}^{n-1}$; in flat Euclidean space, we have $S=0$. 

The role of $\rho$ is central to this paper. We summarize some of its properties below.
\begin{lemma}\label{lem:rhodist}
Let $\Sigma \subsetneq \mathbb{S}^{n-1}$ be a Lipschitz domain. Then, there exists a unique function $\rho \in C^\infty(\Sigma) \cap \mathrm{Lip}(\Sigma)$, positive in $\Sigma$ and satisfying \eqref{eq:rho}-\eqref{eq:rhoB}. Moreover,
\begin{equation}\label{eq:COMP}
  c_1 \leq \frac{\rho}{d_{\Sigma}} \leq c_2 \quad \text{in } \Sigma,
\end{equation}
where $c_1$ and $c_2$ are positive constants depending only on $n$ and $\Sigma$.
\end{lemma} 
We note that the proof of \eqref{eq:COMP} follows from \eqref{eq:xiBLOWUP} and the assertion $\rho\in\mathrm{Lip}(\Sigma)$ follows from \eqref{eq:xigradientbound} with $k=1$. If we assume higher boundary regularity, we also have the following information on the gradient.
\begin{lemma}\label{lem:rhoGRAD}
Let $\Sigma \subset \mathbb{S}^{n-1}$ be a $C^{1,\alpha}$ domain for some $\alpha \in (0,1)$, and let $\rho \in C^\infty(\Sigma) \cap \mathrm{Lip}(\Sigma)$ be the positive solution of \eqref{eq:rho}-\eqref{eq:rhoB} in $\Sigma$. Then $\rho \in C^{1,\alpha}(\overline{\Sigma})$, and
\[
|\nabla_\theta \rho| = 1 \quad \text{on } \partial\Sigma.
\]
\end{lemma}

Next, we turn our attention to the equation \eqref{eq:CYLINDRICALLN} linearized at $\xi$. Take a positive solution $u\in C^\infty(V\cap B_1)$ of \eqref{eq:LNFINITECONE}-\eqref{eq:LNFINITECONEB} and the unique positive solution $u_V\in C^\infty(V)$ of \eqref{eq:LNINFCONE}-\eqref{eq:LNINFCONEB}. For any $x\in\mathbb{R}^n\setminus \{0\}, $ in cylindrical coordinates $t=-\ln(r)$ and $r=|x|$, define, for $(t,\theta)\in\mathbb{R}\times\Sigma$,
\begin{equation}
v(t,\theta)=|x|^{\frac{n-2}{2}}(u-u_V)=|x|^{\frac{n-2}{2}}u-\xi(\theta)=|x|^{\frac{n-2}{2}}u-\rho^{-\frac{n-2}{2}}(\theta).
\end{equation}
Following \cite{HJS2024}, a straightforward computation yields
\begin{equation}\label{eq:LvFv}
\mathcal{L}v=F(v),
\end{equation}
where
\begin{equation}\label{eq:ELLv}
\mathcal{L}v=\partial_{tt}v+\Delta_\theta v-\frac{1}{4}(n-2)^2v-\frac{1}{4}n(n+2)\frac{v}{\rho^2},
\end{equation}
and 
\begin{equation}
F(v)=\rho^{\frac{n-6}{2}}v^2h(\rho^{\frac{n-2}{2}}v),
\end{equation}
where $h\in C^\infty ((-1,1))$ and $h(0)=0$. In fact, 
\begin{equation}
h(s)=\frac{1}{4}n(n-2)s^{-2}\big[(1+s)^{\frac{n+2}{n-2}}-1-\frac{n+2}{n-2}s\big].
\end{equation}
We are interested in the spectral decomposition of \eqref{eq:ELLv}. To this end, we define, for a fixed positive constant $\kappa$ and for any $u\in C^2(\Sigma)$
\begin{equation}\label{eq:ANGOP}
  L u
  = \Delta_\theta u
  - \frac{\kappa}{\rho^2}\,u.
\end{equation}
Observe that $L$ is a linear operator $\Sigma$ with a singular zeroth order coefficient by \eqref{eq:rhoB}. Next, we recall the existence of weak solutions of $-Lu=f$. The proof of the following result relies on Hardy's inequality. See \cite{HJS2024} for details.
\begin{lemma}\label{lem:HARDY}
Let $\Sigma\subsetneq \mathbb{S}^{n-1}$ be a Lipschitz domain and $\rho\in C^\infty (\Sigma)\cap \text{Lip}(\Sigma)$ be a positive function in $\Sigma$ satisfying \eqref{eq:rho}-\eqref{eq:rhoB}. Then, for any $f\in L^2(\Sigma)$, there exists a unique weak solution $u\in H^1_0(\Sigma)$ of $Lu=-f$, and 
\begin{equation*}
    {\lVert \nabla_\theta u\rVert}_{L^2(\Sigma)}+{\lVert \rho^{-1}u\rVert}_{L^2(\Sigma)} \leq C{\lVert f\rVert}_{L^2(\Sigma)},
\end{equation*}
\textit{where $C$ is a positive constant depending only on $n$ and $\Sigma$. Moreover, $v\in C^\infty(\Sigma)$.}
\end{lemma}
By the compact embedding 
$H_0^1(\Sigma) \hookrightarrow L^2(\Sigma)$, we deduce the following result for the eigenvalue problem associated to $L$.
\begin{theorem}\label{thm:eigenvalues}
  Let $\Sigma\subsetneq\mathbb{S}^{n-1}$ be a Lipschitz domain and $\rho\in C^\infty(\Sigma)\cap \mathrm{Lip}(\Sigma)$ a positive function satisfying \eqref{eq:rho}-\eqref{eq:rhoB}. Then, there exists an increasing sequence of positive constants $\{\lambda_i\}_{i\ge1}\to\infty$ and an $L^2(\Sigma)$–orthonormal basis $\{\phi_i\}_{i\ge1}\text{ in }H^1_0(\Sigma)\cap C^\infty(\Sigma)\text{ solving }L\phi_i=-\lambda_i\phi_i$ weakly.
\end{theorem}
We note that the first eigenvalue $\lambda_1$ has multiplicity 1. Furthermore, the following Fredholm alternative holds.
\begin{theorem}\label{thm:Fredholm}
Let $\Sigma\subsetneq\mathbb{S}^{n-1}$ be a Lipschitz domain and 
$\rho\in C^\infty(\Sigma)\cap\mathrm{Lip}(\Sigma)$ a positive function 
satisfying \eqref{eq:rho}-\eqref{eq:rhoB}.  Let 
$\{\lambda_i\}_{i\ge1}$ and $\{\phi_i\}_{i\ge1}$ be as in Theorem \ref{thm:eigenvalues}.  
Then the following hold.
\begin{enumerate}
  \item[\textup{(i)}] 
    For any $\lambda\notin\{\lambda_i\}_{i\geq 1}$ and any $f\in L^2(\Sigma)$, there exists a unique weak solution
    $u\in H^1_0(\Sigma)$ of
    \begin{equation}\label{eq:Llambda}
      L u + \lambda\,u = f \quad\text{in }\Sigma.
    \end{equation}
    Moreover,
    \begin{equation}\label{eq:L2angOP}
      \|u\|_{H^1_0(\Sigma)} \;\le\; C\,\|f\|_{L^2(\Sigma)}.
    \end{equation}
    \item[\textup{(ii)}]
        For any $\lambda=\lambda_i$ for some $i$ and any $f\in L^2(\Sigma)$ with
        \((f,\phi_k)_{L^2(\Sigma)}=0\) for every eigenfunction $\phi_k$ corresponding to the
        eigenvalue $\lambda_i$, there exists a unique weak solution $u\in H^1_0(\Sigma)$ of
        \eqref{eq:Llambda}, satisfying \((u,\phi_k)_{L^2(\Sigma)}=0\) for every eigenfunction $\phi_k$
        corresponding to $\lambda_i$. Moreover, \eqref{eq:L2angOP} holds.
\end{enumerate}
\end{theorem}
In the rest of this paper, we fix \begin{equation}
\kappa=\frac{1}{4}n(n+2),
\end{equation}
and set
\begin{equation}
\beta = \frac{1}{2}(n-2).
\end{equation}
Owing to \cite{HJS2024}, we have the following decay estimate near $\partial\Sigma$.
\begin{theorem}\label{thm:weightedEV}
Let $\Sigma \subsetneq \mathbb{S}^{n-1}$ be a domain with $\partial\Sigma\in C^{1,\alpha}$ for some $\alpha\in(0,1)$, and
$\rho \in C^\infty(\Sigma)\cap \mathrm{Lip}(\Sigma)$ be the positive solution of \eqref{eq:rho}-\eqref{eq:rhoB}.
Assume $s>0$ satisfies
\begin{equation}
  s(s-1) = \kappa,
\end{equation}
and that, for some $\lambda \in \mathbb{R}$ and $f \in C^\infty(\Sigma)$, the function $u \in C^\infty(\Sigma)$ is a bounded solution of \eqref{eq:Llambda}.
If, for some $A>0$ and $a>0$ with $a \neq s$,
\begin{equation}
  |f| \,\le\, A\,\rho^{\,a-2}\quad \text{in }\Sigma,
\end{equation}
then, with $b := \min\{a,s\}$,
\begin{equation}\label{eq:414}
  |u| + \rho\,|\nabla_\theta u| \,\le\, C\bigl(\,\|u\|_{L^2(\Sigma)} + A\,\bigr)\,\rho^{\,b}
  \quad \text{in }\Sigma,
\end{equation}
where $C$ is a positive constant depending only on $n,\alpha,c_0,\kappa,a,\lambda$, and $\Sigma$.

In particular, let $\{\phi_i\}$ be the eigenfunctions from Theorem \ref{thm:eigenvalues} with eigenvalue $\lambda_i$.
Then, $\phi_i \in \mathrm{Lip}(\Sigma)$ and
\begin{equation}
  |\phi_i| + \rho\,|\nabla_\theta \phi_i| \,\le\, C_i\,\rho^{\,s}\quad \text{in }\Sigma,
\end{equation}
where $C_i$ is a positive constant depending only on $n,\kappa,\alpha,c_0,\lambda_i$, and $\Sigma$.
\end{theorem}
We note the choice of notation $C_i$ in the above theorem to emphasize that $C_i$ may not be bounded as $i\rightarrow\infty$. 
It is crucial to point out that with our choice of $\kappa=\frac{1}{4}n(n+2)$, we have 
\begin{equation}\label{eq:s}
s=\frac{1}{2}(n+2).
\end{equation}
In particular, since $n\geq 3$, we have $s>2$.

By Theorem \ref{thm:eigenvalues}, let $\{\lambda_i\}_{i\geq 1}$ be an increasing sequence of positive eigenvalues diverging to $\infty$, and $\{\phi_i\}_{i\geq1}$ an $L^2(\Sigma)$-orthonormal basis in $H^1_0(\Sigma)\cap C^\infty(\Sigma)$  satisfying $L\phi_i=-\lambda_i\phi_i$. Rewriting \eqref{eq:ELLv} in terms of \eqref{eq:ANGOP}, we have
\begin{equation}\label{eq:mathcalL}\mathcal{L}v=\partial_{tt}v+Lv-\beta^2v.
\end{equation}
For each fixed $i\geq 1$ and for some arbitrary function $\psi=\psi(t)\in C^2(\mathbb{R})$, we compute:
\begin{equation}
    \mathcal{L}(\psi\phi_i)=(L_i\psi)\phi_i,
\end{equation}
where the operator $L_i$ is defined by
\begin{equation}\label{eq:defLi}
    L_i\psi=\psi_{tt}-(\lambda_i+\beta^2)\psi.
\end{equation}
The kernel of $L_i$, denoted Ker($L_i$), is spanned by  $e^{-\gamma_it}$ and $e^{\gamma_it}$, with
\begin{equation}\label{eq:gammai}
    \gamma_i=\sqrt{\lambda_i+\beta^2}. 
\end{equation} We point out that $\{\gamma_i\}_{i\geq 1}$ is increasing and diverges to $\infty$.
For future reference, we define the associated  \textit{index set} $\mathcal{I}$ by 
\begin{equation}\label{eq:INDEX}
    \mathcal{I}=\{\sum_{i\geq1}m_i\gamma_i; m_i\in\mathbb{Z}_+ \,\text{with finitely many } m_i>0 \}.
\end{equation}
Equivalently, $\mathcal{I}$ consists of finite linear combinations of $\gamma_1,\gamma_2,\cdots$ with positive integer coefficients. Note that some $\gamma_i$ may be written as a linear combination of $\gamma_1,\cdots,\gamma_{i-1}$ with positive integer coefficients whose sum is at least two.

As a final preparation, we define a special choice of boundary flattening coordinates. Writing $(y',y_{n-1})$ with $y'=(y_1,\cdots,y_{n-2})$ for the coordinates of $\mathbb{R}^{n-1}$, since $\Sigma$ is a bounded $C^{2,\alpha}$ domain, the compactness of $\partial\Sigma$ guarantees the existence of $r>0$ such that for any $\theta_0\in\partial\Sigma$, we may choose a $C^{2,\alpha}$ diffeomorphism $\Psi$ satisfying
\begin{equation}\label{eq:boundaryflat}
\begin{split}
\Psi(\Sigma\cap B_{r}(\theta_0)) &\subset \{(y',y_{n-1}) : y_{n-1}>0\}, \\
\Psi(\partial\Sigma\cap B_{r}(\theta_0)) &\subset \{(y',y_{n-1}) : y_{n-1}=0\}.
\end{split}
\end{equation}
Moreover, we may arrange that $y_{n-1}=\rho$ is the last coordinate. In these coordinates, for a sufficiently small $\delta_0>0$, the image always contains a rectangle of the form
\begin{equation*}
    G_{\delta_0}=\{(y',y_{n-1})\in\mathbb{R}^{n-1} \colon|y'|<\delta_0, \quad 0<y_{n-1}<\delta_0\}.
\end{equation*}
Each point in the collar neighborhood $\Sigma_{\delta_0}\coloneq\{\theta\in\Sigma : \rho(\theta)<\delta_0\}$ is represented by exactly one $(y',y_{n-1})$ in one of finitely many such charts; for simplicity, we write $G_{\delta_0}$ to denote the image of $\Sigma_{\delta_0}$ under the boundary flattening diffeomorphism $\Psi$. Note that we still write $\Sigma_{\delta_0}$ to reference the collar neighborhood in the original geometry.
For this fixed $\delta_0>0$, we work in one of finitely many boundary coordinate charts $G_{\delta_0}$. 

We define for each fixed $i\geq 1$ and $(T,Y',N)\in[t_0,\infty)\times G_{{\delta_0}}$, the \textit{degenerate cylinders} $Q_i=Q_i(T,Y',N)$ as
\begin{equation}\label{eq:Qi}
Q_i(T,Y',N)=\begin{cases}
\big[t_0,t_0+\frac{N}{2^i}\big)\times B_{\frac{N}{2^i}}(Y')\times(N-\frac{N}{2^i},N+\frac{N}{2^i}) & \text{for } t_0\leq T< t_0+\frac{N}{2^i}, \\[1mm]

\big(T-\frac{N}{2^i},T+\frac{N}{2^i}\big)\times B_{\frac{N}{2^i}}(Y')\times (N-\frac{N}{2^i},N+\frac{N}{2^i}) & \text{for } T\geq t_0+\frac{N}{2^i}.
\end{cases}
\end{equation}
We note that, without loss of generality, we can assume the above definition makes sense by replacing $\delta_0$ with $\delta_0/2$ to ensure each cylinder lies within a boundary chart. For convenience, write
\begin{equation}
G^\delta_{\delta_0}=\{(y',y_{n-1})\in\mathbb{R}^{n-1} \colon|y'|<\delta_0, \quad 0<y_{n-1}<\min\{\delta_0,\delta\}\}.
\end{equation}
For any $\delta>0$ such that $0<\delta<\delta_0$, we denote by $\mathcal{Q}_{i,\delta}$ the collection of degenerate cylinders at scale $i\geq 1$ with maximum center height $\delta>0$, defined by
\begin{equation}
    \mathcal{Q}_{i,\delta}=\{Q_i(T,Y',N): (T,Y',N)\in [t_0,\infty)\times G^\delta_{\delta_0} \}.
\end{equation}
\par\indent We now introduce weighted Hölder spaces on $[t_0,\infty)\times\Sigma$. Fix a $t_0>0$. For nonnegative integers $i$ and $k$, $\alpha\in(0,1)$, $\mu, s\in\mathbb{R}$, and $0<\delta<\delta_0$, set 
\begin{align*}
    {\lVert w\rVert}_{\Lambda_\mu^{k}([t_0,\infty)\times\Sigma\setminus \Sigma_\delta)}&=\sum\limits_{j=0}^k \sup\limits_{(t,\theta)\in[t_0,\infty)\times(\Sigma\setminus \Sigma_\delta)} e^{\mu t}|\nabla^jw(t,\theta)|,
    \\
    \\
    {\lVert w\rVert}_{\Lambda_\mu^{k,\alpha}([t_0,\infty)\times\Sigma\setminus \Sigma_\delta)}&={\lVert w\rVert}_{\Lambda_\mu^{k}([t_0,\infty)\times\Sigma\setminus \Sigma_\delta)}+\sup\limits_{t\geq t_0+1 } e^{\mu t}[\nabla^kw]_{C^\alpha\big([t-1,t+1]\times (\Sigma\setminus \Sigma_\delta)\big)},
\end{align*}
and 
\begin{align*}
    {\lVert w\rVert}_{\Lambda_{\mu,s}^{k}([t_0,\infty)\times\Sigma_\delta)}&=\sum\limits_{j=0}^k \sup\limits_{(T,Y',N)\in[t_0,\infty)\times G^\delta_{\delta_0}} e^{\mu T}N^{-s+j}|\nabla^jw(T,Y',N)|,
    \\
    \\
    {\lVert w\rVert}_{\Lambda_{\mu,s}^{k,\alpha}([t_0,\infty)\times\Sigma_\delta)}&={\lVert w\rVert}_{\Lambda_{\mu,s}^{k}([t_0,\infty)\times\Sigma_\delta)}
    \\&\quadd+\sup\limits_{(T,Y',N)\in[t_0,\infty)\times G^\delta_{\delta_0}}e^{\mu T}N^{-s+k+\alpha}[\nabla^k w]_{C^\alpha\big(Q_i(T,Y',N)\big)},
\end{align*}
where $[\,\cdot\,]_{C^\alpha}$ is the usual Hölder semi-norm. The supremum in the final term is taken over all points $(T,Y',N)\in G^\delta_{\delta_0}$. These norms are adapted to the degeneracy near $\partial\Sigma$, and the choice of cylinders $Q_i$ ensures compatibility with local estimates. Moreover, define, for $0<\delta'<\delta''<\delta$,
\begin{equation*}
\|w\|_{\Lambda_{\mu,s}^{k,\alpha}([t_0,\infty)\times\Sigma)} 
= \|w\|_{\Lambda_\mu^{k,\alpha}([t_0,\infty)\times (\Sigma\setminus \Sigma_{\delta'}))} 
+ \|w\|_{\Lambda_{\mu,s}^{k,\alpha}([t_0,\infty)\times \Sigma_{\delta''})}.
\end{equation*}
\begin{defn}
We define the weighted Hölder space $\Lambda_{\mu,s}^{k,\alpha}([t_0,\infty)\times\Sigma)$ 
to be the set of functions $w\text{ in } C^{k,\alpha}_{\mathrm{loc}}([t_0,\infty)\times\Sigma)$
with a finite $\|w\|_{\Lambda_{\mu,s}^{k,\alpha}([t_0,\infty)\times\Sigma)}$.
\end{defn}
We note that the choice of specific $i\geq 1$ and appropriate $\delta',\delta''>0$ in the above definitions result in equivalent norms.
Similarly, we can define corresponding weighted Hölder spaces on $\Sigma$ for function $\varphi=\varphi(\theta)$ that are independent of $t$. 
In particular, let
\begin{equation*}
    Q_i(Y',N)=B_{\frac{N}{2^i}}(Y')\times(N-\frac{N}{2^i},N+\frac{N}{2^i}),
\end{equation*}
and define
\begin{align*}
    {\lVert \varphi\rVert}_{\Lambda_{s}^{k}(\Sigma_\delta)}&=\sum\limits_{j=0}^k \sup\limits_{(Y',N)\in G^\delta_{\delta_0}} N^{-s+j}|\nabla^j \varphi(T,Y',N)|,
    \\
    {\lVert \varphi\rVert}_{\Lambda_{s}^{k,\alpha}(\Sigma_\delta)}&={\lVert \varphi\rVert}_{\Lambda_{s}^{k}(\Sigma_\delta)}+\sup\limits_{(Y',N)\in G^\delta_{\delta_0}}N^{-s+k+\alpha}[\nabla^k \varphi]_{C^\alpha\big(Q_i(Y',N)\big)}.
\end{align*}
Similarly, define, for $0<\delta'<\delta''<\delta$, 
\begin{equation*}
{\lVert \varphi\rVert}_{\Lambda_{s}^{k,\alpha}(\Sigma)}={\lVert \varphi\rVert}_{C^{k,\alpha}(\Sigma\setminus\Sigma_{\delta'})}+{\lVert \varphi\rVert}_{\Lambda_{s}^{k,\alpha}(\Sigma_{\delta''})}.
\end{equation*}
We define the corresponding weighted Hölder space $\Lambda_{s}^{k,\alpha}(\Sigma)$ as the set of all functions $\varphi\in C^{k,\alpha}_{\text{loc}}(\Sigma)$ with a finite ${\lVert \varphi\rVert}_{\Lambda_{s}^{k,\alpha}(\Sigma)}$.

Recall the linear operator $\mathcal{L}$ defined by \eqref{eq:ELLv}, $\mu>0$ and $s>2$. For some $f\in \Lambda_{\mu,s-2}^{0,\alpha}([t_0,\infty)\times\Sigma)$, consider the linear equation 
\begin{equation}\label{eq:OGLINEAREQ}
    \mathcal{L}v=f \quad\text{in } (t_0,\infty)\times\Sigma. 
\end{equation}
To ensure uniqueness of solutions, we impose an appropriate boundary condition at $t=t_0$. In this way, the operator
\begin{equation*}    \mathcal{L}:\Lambda_{\mu,s}^{2,\alpha}([t_0,\infty)\times\Sigma)\rightarrow \Lambda_{\mu,s-2}^{0,\alpha}([t_0,\infty)\times\Sigma)
\end{equation*}
admits a bounded inverse.
To this end, consider the Dirichlet boundary-value problem
\begin{equation}\label{eq:OGDIRICH}
\begin{aligned}
   \mathcal{L}v &= f \quad \text{in } (t_0,\infty)\times\Sigma, \\
   v &= \varphi \quad \text{on } \{t_0\}\times\Sigma.
\end{aligned}
\end{equation}

For later purposes, it is convenient to introduce a transformed version of \eqref{eq:OGDIRICH}. Since the zeroth order coefficient of $\mathcal{L}v$ is singular in $\rho^2$, we set 
\begin{equation*}
v=\rho^2w
\end{equation*}
and compute
\begin{equation}
\mathcal{L}(\rho^2 w)
=\rho^2(w_{tt}+\Delta_\theta w)+4\rho\nabla_\theta\rho\cdot \nabla_\theta w +[-\frac{1}{4}(n-2)^2\rho^2-\frac{1}{4}n(n+2)+2\rho\Delta_\theta \rho +2|\nabla_\theta\rho|^2]w.
\end{equation}
For convenience, define
\begin{equation}\label{eq:deffrakL}
    \mathfrak{L}w=\rho^2w_{tt}+\rho^2 a_{ij}\partial_{ij}w+\rho b_i \partial_i w +cw,
\end{equation}
where 
\begin{equation*}
    \begin{cases}
    a_{ij}=\delta_{ij}, \\[1mm]
    b_i=4\partial_i \rho, \\[1mm]
    c=-\frac{1}{4}(n-2)^2\rho^2-\frac{1}{4}n(n+2)+2\rho\Delta_\theta \rho +2|\nabla_\theta\rho|^2.
    \end{cases}
\end{equation*}
Given boundary data $\varphi$ on $\{t_0\}\times\Sigma$ as in \eqref{eq:OGDIRICH}, we define the corresponding transformed boundary data
\begin{equation}
\psi =\frac{\varphi}{\rho^2}\quad\text{on }\{t_0\}\times\Sigma.
\end{equation}
The transformed Dirichlet problem for $w$ reads
\begin{equation}
\begin{aligned}\label{eq:TRANSDIRICH}
   \mathfrak{L}w &= f \quad \text{in } (t_0,\infty)\times\Sigma, \\
   w &= \psi \quad \text{on } \{t_0\}\times\Sigma,
\end{aligned}
\end{equation}
where $f$ is the same right-hand side as before in \eqref{eq:OGDIRICH}.
Note that the linear operator in \eqref{eq:deffrakL} is \emph{uniformly degenerate} on $(t_0,\infty)\times\partial\Sigma$, and by Lemma \ref{lem:rhoGRAD}, $c<0$ on $\partial\Sigma$ independent of $t$. Nonetheless, uniqueness of solutions holds. Recall that $\{\phi_i\}_{i\geq 1}$ is a fixed sequence of eigenfunctions on $\Sigma$, which form an orthonormal basis in $L^2(\Sigma)$, and that $\{\gamma_i\}_{i\geq1}$ is the sequence given by \eqref{eq:gammai}.

\begin{lemma}\label{lem:Uniqueness} Let $ \mu>\gamma_1$, and $\varphi \in C(\Sigma)$. Then, the boundary value problem \eqref{eq:OGDIRICH} admits at most one solution $v\in C^2_{\mu}([t_0,\infty)\times\Sigma)$.
\end{lemma}
\begin{proof}
Suppose $f=0$ and $\varphi=0$, and let $v\in C^2_{\mu}([t_0,\infty)\times\Sigma)$ solve the homogeneous version of \eqref{eq:OGDIRICH}. For each $i\geq 1$,  define the projection
\begin{equation*}
    v_i(t)=\int_{\Sigma} v(t,\theta)\phi_i(\theta)\,d\theta.
\end{equation*}
Then, $v_i(t)$ satisfies the ODE $L_iv_i=0$ on $(t_0,\infty)$ with initial condition $v_i(t_0)=0$.
As $\text{Ker}(L_i)=\operatorname{span}\{e^{-\gamma_i t},e^{\gamma_i t}\}$, we have 
\begin{equation*}
    v_i(t)=c_{i,1}e^{-\gamma_it}+c_{i,2}e^{\gamma_it},
\end{equation*}
for some constants $c_{i,1},c_{i,2}\in \mathbb{R}$. Since $v\in C^2_{\mu}([t_0,\infty)\times\Sigma)$, we have $|e^{\mu t}v_i(t)|\leq C$ as $t\to\infty$, which forces $v_i=0$ when $\gamma_i<\mu$ and $c_{i,2}=0$ when $\gamma_i\geq\mu$. In either case, $v_i(t_0)=0$ yields $c_{i,1}=0$, and hence $v_i=0$ for all $i\geq 1$, and therefore $v=0$.
\end{proof}

Next, we establish a  $C^{2,\alpha}$-estimate of solutions of \eqref{eq:OGDIRICH}.

\begin{lemma}\label{lem:Apriori}
\textit{Let $\alpha\in(0,1),\, \mu>0,\, s>2,\, f\in \Lambda_{\mu,s-2}^{0,\alpha}([t_0,\infty)\times\Sigma),  \text{ and } \varphi \in \Lambda_{s}^{2,\alpha}(\Sigma)$. Suppose $v\in \Lambda_{s}^{2,\alpha}([t_0,\infty)\times\Sigma)$ is a solution of \eqref{eq:OGDIRICH}. Then},
\begin{equation}\label{eq:Apriori}
    {\lVert v\rVert}_{\Lambda_{\mu,s}^{2,\alpha}([t_0,\infty)\times\Sigma)} 
    \leq C\{{\lVert v\rVert}_{\Lambda_{\mu,s}^{0}([t_0,\infty)\times\Sigma)}+{\lVert f\rVert}_{\Lambda_{\mu,s-2}^{0,\alpha}([t_0,\infty)\times\Sigma)}
    +e^{\mu t_0}{\lVert \varphi\rVert}_{\Lambda_{s}^{2,\alpha}(\Sigma)}\},
\end{equation}
\textit{where $C$ is a positive constant depending only on $n,\alpha,\mu$,  and the geometry of  $\Sigma$, independent of $t_0$.}
\end{lemma}

\begin{proof}
Recall $v=\rho^2 w$. In view of \eqref{eq:TRANSDIRICH}, and the continuity of multiplication by \(\rho^2\) from \(\Lambda_{s}^{k,\alpha}(\Sigma)\) to \(\Lambda_{s+2}^{k,\alpha}(\Sigma)\), and similarly for \(\Lambda_{\mu,s}^{k,\alpha}([t_0,\infty)\times\Sigma)\) to \(\Lambda_{\mu,s+2}^{k,\alpha}([t_0,\infty)\times\Sigma)\), it suffices to establish
\begin{equation*}
    \| w \|_{\Lambda_{\mu,s-2}^{2,\alpha}([t_0,\infty)\times\Sigma)} \leq C \{
        \| w \|_{\Lambda_{\mu,s-2}^{0}([t_0,\infty)\times\Sigma)}
        + \| f \|_{\Lambda_{\mu,s-2}^{0,\alpha}([t_0,\infty)\times\Sigma)}
        + e^{\mu t_0} \| \psi \|_{\Lambda_{s-2}^{2,\alpha}(\Sigma)}
    \},
\end{equation*}
where $\psi=\tfrac{\varphi}{\rho^2}\in  {\Lambda_{s-2}^{2,\alpha}(\Sigma)}$.
We split the domain into  the collar $\Sigma_\delta$ and its complement, derive Schauder estimates on each piece, and then patch them together. To this end, we first establish the following local estimates. Consider $0<\delta_1<\delta_2<\delta_3<\delta_4 < \delta_0$. We claim
\begin{equation}\label{eq:uno}
\begin{aligned}
&\| w \|_{\Lambda_{\mu,s-2}^{2,\alpha}([t_0,\infty)\times\Sigma_{\delta_3})} 
\\
&\quadd\leq C \{
\| w \|_{\Lambda_{\mu,s-2}^{0}([t_0,\infty)\times\Sigma_{\delta_4})}
+ \| f \|_{\Lambda_{\mu,s-2}^{0,\alpha}([t_0,\infty)\times\Sigma_{\delta_4})}
+ e^{\mu t_0} \| \psi \|_{\Lambda_{s-2}^{2,\alpha}(\Sigma_{\delta_4})}\},
\end{aligned}
\end{equation}
and
\begin{equation}\label{eq:dos}
\begin{aligned}
&\| w \|_{\Lambda_\mu^{2,\alpha}([t_0,\infty)\times(\Sigma \setminus \Sigma_{\delta_2}))} 
\\
&\quadd\leq C \{
\| w \|_{\Lambda_\mu^{0}([t_0,\infty)\times(\Sigma \setminus \Sigma_{\delta_1}))}
+ \| f \|_{\Lambda_\mu^{0,\alpha}([t_0,\infty)\times(\Sigma \setminus \Sigma_{\delta_1}))}
+ e^{\mu t_0} \| \psi \|_{C^{2,\alpha}(\Sigma \setminus \Sigma_{\delta_1})}
\}.
\end{aligned}
\end{equation}
Adding the estimates in \eqref{eq:uno} and \eqref{eq:dos} yields the desired result.

To prove \eqref{eq:uno}, given our fixed $\delta_0>0$, choose $0<\delta_3<\delta_4<\delta_0$ so that if $Q_3(T,Y',N)\in\mathcal{Q}_{3,\delta_3}$, then $Q_1(T,Y',N)\in\mathcal{Q}_{1,\delta_4}$. We aim to cover the set $Q_3(T,Y',N)$ by regions on which standard estimates apply. Now, fixing this particular choice of $(T,Y',N)$, we proceed by splitting into several cases.

{\it Case 1.} We consider $t_0\leq T< t_0+\frac{N}{2}$. One easily checks that
\begin{equation*}
    Q_3(T,Y',N)\subset  Q_1(t_0,Y',N) \cup Q_1(t_0+\frac{N}{2},Y',N),
\end{equation*}
and, in fact,
\begin{equation*}
    Q_1(t_0,Y',N) \cup Q_1(t_0+\frac{N}{2},Y',N)=[t_0,t_0+N)\times B_{\frac{N}{2}}(Y')\times(N-\frac{N}{2},N+\frac{N}{2}),
\end{equation*} 
as required. 
We consider the equation
\begin{equation}
\begin{aligned}
    \mathfrak{L}w=y^2_{n-1} &w_{tt}+y^2_{n-1} a_{ij}\partial_{ij}w+y_{n-1} b_i \partial_i w +cw = f \\ &\quad\text{in } (t_0,t_0+N)\times B_{\frac{N}{2}}(Y')\times(N-\frac{N}{2},N+\frac{N}{2}),
\end{aligned}
\end{equation}
where the boundary data is given by
\begin{equation*}
w = \psi \quad \text{on } \{t_0\}\times B_{\frac{N}{2}}(Y')\times(N-\frac{N}{2},N+\frac{N}{2}).
\end{equation*}

We proceed by rescaling the domain near $(T,Y',N)$ to a standard cylinder and applying the standard Schauder theory. Specifically, define
\[
(t, y', y_{n-1}) \overset{\Phi}{\longmapsto} (\tau,x',x_{n-1}) = \left( \frac{t-t_0}{N/2}, \frac{y'-Y'}{N/2},\, \frac{y_{n-1}-N}{N/2} \right),
\]
and observe $\Phi(Q_3(T, Y', N))$ is contained in $[0,2) \times B_1(0) \times (-1,1)$, the rescaled image of the union $Q_1(t_0, Y', N) \cup Q_1(t_0+\frac{N}{2}, Y', N)$. Letting $\omega = w \circ \Phi^{-1}$, the equation becomes
\begin{equation*}
\begin{aligned}
    \tilde{\mathfrak{L}}\omega=4(1+&\frac{x_{n-1}}{2})^2  \omega_{\tau\tau}+4(1+\frac{x_{n-1}}{2})^2 \tilde{a}_{ij}\tilde\partial_{ij}\omega+(1+\frac{x_{n-1}}{2})  \tilde{b}_i \tilde\partial_i \omega +\tilde{c}\omega = \tilde{f}
    \\
    &\text{in } (0,2)\times B_{1}(0)\times(-1,1),
\end{aligned}
\end{equation*}
where the boundary data is given by
\[
\omega = \tilde{\psi} \qquad \text{on } \{\tau = 0\} \times B_1(0) \times (-1,1).
\]
\noindent
The operator $\tilde{\mathfrak{L}}$ is uniformly elliptic, so by the boundary Schauder estimate, for any $(\tau,x',x_{n-1})\in \Phi(Q_3(T,Y',N))$, we have
\begin{equation*}
\begin{aligned}
&\sum_{j=0}^2|\nabla^j\omega(\tau,x',x_{n-1})|+[\nabla^2\omega]_{C^\alpha(\Phi(Q_3(T,Y',N))))}
\\
&\quadd\quadd\leq C\{{\lVert \omega\rVert}_{L^\infty([0,2)\times B_{1}(0)\times(-1,1))}+{\lVert \tilde{f}\rVert}_{C^{0,\alpha}([0,2)\times B_{1}(0)\times(-1,1))}+{\lVert \tilde{\psi}\rVert}_{C^{2,\alpha}(B_{1}(0)\times(-1,1)}\}.
\end{aligned}
\end{equation*}
By rescaling by $\Phi^{-1}$, multiplying by $e^{\mu T}N^{-s+2}$, evaluating at $(t,y',y_{n-1})=(T,Y',N)$, and taking the supremum over $(T,Y',N) \in[t_0,t_0+\frac{N}{2})\times G^{\delta_3}_{\delta_0}$, we obtain
\begin{equation}\label{eq:big-estimate1}
\begin{aligned}
  &\sup_{(T,Y',N) \in[t_0,t_0+\frac{N}{2})\times G^{\delta_3}_{\delta_0}}\{\sum_{j=0}^2e^{\mu T}N^{-s+2+j}|\nabla^j w(T,Y',N)|+
e^{\mu T}N^{-s+4+\alpha}[\nabla^2 w]_{C^\alpha(Q_3(T,Y',N))}\}
\\
&\quad\quad\quadd\leq C \{
        \| w \|_{\Lambda_{\mu,s-2}^{0}([t_0,\infty)\times\Sigma_{\delta_4})}
        + \| f \|_{\Lambda_{\mu,s-2}^{0,\alpha}([t_0,\infty)\times\Sigma_{\delta_4})}
        + e^{\mu t_0} \| \psi \|_{\Lambda_{s-2}^{2,\alpha}(\Sigma_{\delta_4})}
    \}.
\end{aligned}
\end{equation}
Note that since $T\in[t_0,t_0+\frac{N}{2})$ and $N<\delta_0$, we have $e^{\mu T}\leq Ce^{\mu t_0}$, giving us above the bound on the right hand side.

{\it Case 2.}
The case $T\geq t_0+\tfrac{N}{2}$ is similar and easier. By applying the same process as above to the case $Q_3(T,Y',N)\subset Q_1(T,Y',N)$, by the interior Schauder estimate, we have
\begin{equation}\label{eq:big-estimate2}
\begin{aligned}
    &\sup_{(T,Y',N) \in[t_0+\frac{N}{2},\infty)\times G^{\delta_3}_{\delta_0}}\{\sum_{j=0}^2e^{\mu T}N^{-s+2+j}|\nabla^j w(T,Y',N)|+
    e^{\mu T}N^{-s+4+\alpha}[\nabla^2 w]_{C^\alpha(Q_3(T,Y',N))}\}
    \\
  &\quad\quad\quad\leq C \{
        \| w \|_{\Lambda_{\mu,s-2}^{0}([t_0,\infty)\times\Sigma_{\delta_4})}
        + \| f \|_{\Lambda_{\mu,s-2}^{0,\alpha}([t_0,\infty)\times\Sigma_{\delta_4})}
        + e^{\mu t_0} \| \psi \|_{\Lambda_{s-2}^{2,\alpha}(\Sigma_{\delta_4})}
    \}.
\end{aligned}
\end{equation}
Combining \eqref{eq:big-estimate1} and \eqref{eq:big-estimate2}, we have \eqref{eq:uno}.

The proof of \eqref{eq:dos} is similar to the approach taken in \cite{HanLi2020}. For completeness, we include it here.
    On $\Sigma\setminus\Sigma_{\delta_2}\subset\Sigma\setminus\Sigma_{\delta_1}$, there are two cases: $t>t_0+2$ and $t_0\leq t\leq t_0+2$. First, consider $t>t_0+2$.
    Then, we have $[t-1,t+1]\times(\Sigma\setminus\Sigma_{\delta_2})\subset[t-2,t+2]\times(\Sigma\setminus\Sigma_{\delta_1})$.
    By the interior Schauder estimate, we have
\begin{equation*}
\begin{aligned}
  \sum\limits_{j=0}^2 &\sup\limits_{\Sigma\setminus \Sigma_{\delta_2}} |\nabla^jw(t,\cdot)|+[\nabla^2w]_{C^\alpha\big([t-1,t+1]\times (\Sigma\setminus \Sigma_{\delta_2})\big)}
\\
  &\leq C\{{\lVert w\rVert}_{L^\infty([t-2,t+2]\times(\Sigma\setminus\Sigma_{\delta_1})}+{\lVert f\rVert}_{L^\infty([t-2,t+2]\times(\Sigma\setminus\Sigma_{\delta_1})}+[f]_{C^\alpha([t-2,t+2]\times(\Sigma\setminus\Sigma_{\delta_1})}\}.
\end{aligned}
\end{equation*}
    To estimate the Hölder semi-norm of $f$ on $[t-2,t+2]\times(\Sigma\setminus\Sigma_{\delta_1})$ in the right-hand side, we take $(t_1,\theta_1),(t_2,\theta_2)\in [t-2,t+2]\times(\Sigma\setminus\Sigma_{\delta_1})$ with $(t_1,\theta_1)\neq (t_2,\theta_2).$ We consider two cases: $|t_1-t_2|\leq 2$ and $|t_1-t_2|\geq 2$. In the first case, there exists a $t'\in[t-1,t+1]$ such that $t_1,t_2\in[t'-1,t'+1]$. Hence, in either case, we get
    \begin{equation*}
        [f]_{C^\alpha([t-2,t+2]\times(\Sigma\setminus\Sigma_{\delta_1})}
        \leq \max\{
        \sup\limits_{t'\in[t-1,t+1]}[f]_{C^\alpha([t'-1,t'+1]\times(\Sigma\setminus\Sigma_{\delta_1})},{\lVert f\rVert}_{L^\infty([t-2,t+2]\times(\Sigma\setminus\Sigma_{\delta_1}))}\}.
    \end{equation*}
Then,
\begin{equation*}
\begin{aligned}
  &\sum\limits_{j=0}^2 \sup\limits_{\Sigma\setminus \Sigma_{\delta_2}} |\nabla^jw(t,\cdot)|+[\nabla^2w]_{C^\alpha\big([t-1,t+1]\times (\Sigma\setminus \Sigma_{\delta_2})\big)}
  \\
  &\leq C\{{\lVert w\rVert}_{L^\infty([t-2,t+2]\times(\Sigma\setminus\Sigma_{\delta_1})}+{\lVert f\rVert}_{L^\infty([t-2,t+2]\times(\Sigma\setminus\Sigma_{\delta_1})}+\sup\limits_{t'\in[t-1,t+1]}[f]_{C^\alpha([t'-1,t'+1]\times(\Sigma\setminus\Sigma_{\delta_1})}\}.
\end{aligned}
\end{equation*}
Multiplying by $e^{\mu t}$ on both sides and taking the supremum over $t\in(t_0+2,\infty)$, we have 
\begin{equation}\label{eq:big-estimate3}
\begin{aligned}
  \sum\limits_{j=0}^2&\sup\limits_{t\in(t_0+2,\infty)} \sup\limits_{\Sigma\setminus \Sigma_{\delta_2}} e^{\mu t}|\nabla^jw(t,\cdot)|+\sup\limits_{t\in(t_0+2,\infty)}e^{\mu t}[\nabla^2w]_{C^\alpha\big([t-1,t+1]\times (\Sigma\setminus \Sigma_{\delta_2})\big)}
\\
  &\quad\leq C \{
        \| w \|_{\Lambda_\mu^{0}([t_0,\infty)\times(\Sigma \setminus \Sigma_{\delta_1}))}
        + \| f \|_{\Lambda_\mu^{0,\alpha}([t_0,\infty)\times(\Sigma \setminus \Sigma_{\delta_1}))}
        + e^{\mu t_0} \| \psi \|_{C^{2,\alpha}(\Sigma \setminus \Sigma_{\delta_1})}
    \}.
\end{aligned}
\end{equation}

    Next, consider $t_0\leq t\leq t_0+2$. Then, we have $[t_0, t_0+3]\times(\Sigma\setminus\Sigma_{\delta_2})\subset[t_0,t_0+4]\times(\Sigma\setminus\Sigma_{\delta_1})$. By the boundary Schauder estimate, we get
\begin{equation*}
\begin{aligned}
  &\sum\limits_{j=0}^2 \sup\limits_{\Sigma\setminus \Sigma_{\delta_2}} |\nabla^jw(t,\cdot)|+[\nabla^2w]_{C^\alpha\big([t_0,t_0+3]\times (\Sigma\setminus \Sigma_{\delta_2})\big)}
    \\&\quadd\leq C\{{\lVert w\rVert}_{L^\infty([t_0,t_0+4]\times(\Sigma\setminus\Sigma_{\delta_1})}+{\lVert f\rVert}_{L^\infty([t_0,t_0+4]\times(\Sigma\setminus\Sigma_{\delta_1})}
    \\
    &\quadd\quadd\quadd\quadd\quadd+[f]_{C^\alpha([t_0,t_0+4]\times(\Sigma\setminus\Sigma_{\delta_1})}+{\lVert \psi\rVert}_{C^{2,\alpha}(\Sigma\setminus\Sigma_{\delta_1})}\}.
\end{aligned}
\end{equation*}
By a similar argument as above, we get
\begin{equation}\label{eq:big-estimate4}
\begin{aligned}
  \sum\limits_{j=0}^2& \sup\limits_{t\in[t_0,t_0+2]}\sup\limits_{\Sigma\setminus \Sigma_{\delta_2}} |\nabla^jw(t,\cdot)|+\sup\limits_{t\in[t_0+1,t_0+2]}[\nabla^2w]_{C^\alpha([t-1,t+1]\times(\Sigma\setminus\Sigma_{\delta_2})}
\\
  &\quad\leq C \{
        \| w \|_{\Lambda_\mu^{0}([t_0,\infty)\times(\Sigma \setminus \Sigma_{\delta_1}))}
        + \| f \|_{\Lambda_\mu^{0,\alpha}([t_0,\infty)\times(\Sigma \setminus \Sigma_{\delta_1}))}
        + e^{\mu t_0} \| \psi \|_{C^{2,\alpha}(\Sigma \setminus \Sigma_{\delta_1})}
    \}.
\end{aligned}
\end{equation}
Combining \eqref{eq:big-estimate3} and \eqref{eq:big-estimate4} yields \eqref{eq:dos}.
\end{proof}
Next, we estimate the $L^\infty$-norm of solutions of \eqref{eq:OGDIRICH} on cylinders of finite length and with homogeneous boundary conditions. The proof below is based on a rescaling argument adapted from \cite{MAZZEO1}.

\begin{lemma}\label{lem:Linfinity}
Let $s>2$, $\mu>\gamma_1$, and assume that $\mu\neq\gamma_i$ for any $i \geq 2$. Let $t_0$ and $T$ be constants with $t_0 \geq 0$ and $T-t_0\geq 4$, and let $f\in C([t_0,T]\times\Sigma)$. Suppose $v\in C^{2} ([t_0,T]\times\Sigma)$ satisfies $\int_\Sigma v(t,\theta)\phi_i(\theta)\,d\theta=0$ for any $\gamma_i<\mu$ and any $t\in[t_0,T]$, and 
\begin{equation*}
\begin{aligned}
    \mathcal{L}v&=f\quad\text{in }(t_0,T)\times\Sigma,
    \\
    v &=0 \quad \text{on } ((\{t_0\}\cup \{T\})\times\Sigma)\cup ((t_0,T)\times \partial\Sigma).
\end{aligned}
\end{equation*}
Then,
\begin{equation}
    \sup\limits_{(t,\theta)\in[t_0,T]\times\Sigma}\rho^{-s} e^{\mu t}|v(t,\theta)|\leq C\sup\limits_{(t,\theta)\in[t_0,T]\times\Sigma} \rho^{-s+2} e^{\mu t}|f(t,\theta)|,
\end{equation}
where $C$ is a positive constant depending only on $n,\alpha,\mu$, and $\Sigma$, and is independent of $t_0$ and $T$.
\end{lemma}
\begin{proof} 
We prove the result by contradiction. Suppose there exist sequences $\{ t_i\}$, $\{ T_i\}$, $\{v_i\}$, and $\{f_i\}$ with $t_i\geq0$ and $T_i-t_i\geq4$, such that 
\begin{equation}
\begin{aligned}
    \mathcal{L}v_i&=f_i  \quad\text{in }(t_i,T_i)\times\Sigma,
    \\
    v_i &=0 \quad \text{on } ((\{t_i\}\cup \{T_i\})\times\Sigma)\cup ((t_i,T_i)\times \partial\Sigma),
\end{aligned}
\end{equation}
and
\begin{equation}
    \sup\limits_{(t,\theta)\in[t_i,T_i]\times\Sigma} \rho^{-s+2} e^{\mu t}|f_i(t,\theta)|=1,
\end{equation}
\begin{equation}
    \sup\limits_{(t,\theta)\in[t_i,T_i]\times\Sigma} \rho^{-s} e^{\mu t}|v_i(t,\theta)| \to\infty  \quad \text{as } i\to\infty.
\end{equation}
Choose $t^*_i\in(t_i,T_i)$ such that 
\begin{equation*}
    A_i\equiv \sup\limits_{\theta\in\Sigma}\rho^{-s}(\cdot)e^{\mu t^*_i}|v_i(t^*_i,\cdot)|= \sup\limits_{(t,\theta)\in[t_i,T_i]\times\Sigma}  \rho^{-s}(\theta)e^{\mu t}|v_i(t,\theta)|.
\end{equation*}
Then, $A_i\to\infty$ as $i\to\infty$. Define
\begin{equation*}
    \tilde{v}_i(t,\theta)=A^{-1}_ie^{\mu t^*_i}v_i(t+t^*_i,\theta),
\end{equation*}
and
\begin{equation*}
    \tilde{f}_i(t,\theta)=A^{-1}_ie^{\mu t^*_i}f_i(t+t^*_i,\theta).
\end{equation*}
Then,
\begin{equation}\label{eq:mass}
    \sup\limits_{\Sigma}\big(\rho^{-s}(\cdot)|\tilde{v}_i(0,\cdot)|\big)=1 ,
\end{equation}
and, for any $(t,\theta)\in [t_i-t^*_i,T_i-t^*_i]\times\Sigma$,
\begin{equation}\label{eq:vtilde}
    \rho^{-s}(\theta)e^{\mu t}|\tilde{v}_i(t,\theta)|\leq 1,
\end{equation}
and 
\begin{equation}\label{eq:ftilde}
    \rho^{-s+2}(\theta)e^{\mu t}|\tilde{f}_i(t,\theta)|\leq A^{-1}_i . 
\end{equation}
Moreover,
\begin{equation*}
    \mathcal{L}\tilde{v}_i=\tilde{f}_i \quad \text{on }(t_i-t^*_i,T_i-t^*_i)\times\Sigma,
\end{equation*}
specifically,
\begin{equation*}
    \partial_{tt}\tilde{v}_i+\Delta_\theta \tilde{v}_i-\frac{1}{4}(n-2)^2\tilde{v}_i-\frac{1}{4}n(n+2)\frac{\tilde{v}_i}{\rho^2}=\tilde{f}_i \quad \text{on }(t_i-t^*_i,T_i-t^*_i)\times\Sigma.
\end{equation*}
Passing to subsequences, we assume, for some $\tau_-\in\mathbb{R}^-\cup\{-\infty\}$ and $\tau_+\in\mathbb{R}^+\cup\{\infty\}$,
\begin{equation*}
    t_i-t^*_i\rightarrow \tau_-,\quad T_i-t^*_i \rightarrow \tau_+.
\end{equation*}
We claim that $\tau_-<0$ if it is finite, and similarly $\tau_+>0$ if it is finite. Thus, to obtain a contradiction, assume that $\tau_-=\lim\limits_{i\to\infty}t_i-t^*_i\to 0^-$. By \eqref{eq:vtilde}, we get
\begin{equation*}
    |\tilde{v}_i(t,\theta)|\leq C\rho^{s}(\theta)e^{\mu(t^*_i-t_i)} \quad\text{on } [t_i-t^*_i,t_i-t^*_i+2]\times\Sigma.
\end{equation*}
Similarly, by \eqref{eq:ftilde}, we get
\begin{equation*}
    |\tilde{f}_i(t,\theta)|\leq CA^{-1}_i\rho^{s-2}e^{\mu(t^*_i-t_i)} \quad\text{on } [t_i-t^*_i,t_i-t^*_i+2]\times\Sigma.
\end{equation*}
We have 
\begin{equation*}
    \partial_{tt}\tilde{v}_i+\Delta_\theta \tilde{v}_i=\tilde{f}_i+\frac{1}{4}n(n+2)\frac{\tilde{v}_i}{\rho^2}+\frac{1}{4}(n-2)^2\tilde{v}_i \quad \text{on }(t_i-t^*_i,t_i-t^*_i+2)\times\Sigma.
\end{equation*}
Hence,
\begin{equation*}
    |\partial_{tt}\tilde{v}_i+\Delta_\theta \tilde{v}_i|\leq C\rho^{s-2}e^{\mu(t^*_i-t_i)}\quad\text{on } (t_i-t^*_i,t_i-t^*_i+2)\times\Sigma.
\end{equation*}
Consider any $\Sigma'\subset\subset\Sigma$. Then,
\begin{equation*}
    [t_i-t^*_i,t_i-t^*_i+1)\times \Sigma'  \subset\subset [t_i-t^*_i,t_i-t^*_i+2)\times \Sigma.
\end{equation*}
Then, for any $t\in (t_i-t^*_i,t_i-t^*_i+1)$ and $\theta\in\Sigma''\subset\Sigma'$, since $\tilde{v}_i=0$ on $\{t_i-t^*_i\}\times\Sigma$, by the $C^1$-estimate up to the boundary, we have
\begin{equation*}
    |\nabla \tilde{v}_i(t,\theta)|\leq C\sup\limits_{(t_i-t^*_i,t_i-t^*_i+2)\times\Sigma'} |\tilde{f}_i|\leq Ce^{\mu(t^*_i-t_i)}.
\end{equation*}

Now, we need to estimate $|\nabla \tilde{v}_i|$ near $\partial\Sigma$.
In the coordinates provided by $\Psi$ as in \eqref{eq:boundaryflat}, with $v'_i=\tilde{v}_i\circ\Psi^{-1}$ and $f'_i=\tilde{f}_i\circ\Psi^{-1}$, we have
\begin{equation*}           \partial_{tt}v'_i+a_{kl}\partial_{kl}v'_i+b_k\partial_{k}v'_i +cv'_i =f'_i+\frac{1}{4}n(n+2)\frac{v'_i}{y_{n-1}^2} \quad \text{on }(t_i-t^*_i,t_i-t^*_i+2)\times G_{\delta},
\end{equation*}
and the associated operator on $v'_i$ has $C^{0,\alpha}$ coefficients and is uniformly elliptic on $(t_i-t^*_i,t_i-t^*_i+2)\times G_{\delta}$.
Using coordinates $(t,y',y_{n-1})$ where $y_{n-1}=\rho$, fix a point $(T,Y',N)\in (t_i-t^*_i,t_i-t^*_i+2)\times G_{\delta}$. 

We first deal with the case when $t_i-t^*_i<T < t_i-t^*_i+\frac{N}{4}$. Consider $[t_i-t^*_i,t_i-t^*_i+\frac{N}{4})\times B_{\frac{N}{4}}(Y')\times (N-\frac{N}{4},N+\frac{N}{4})$ and the transformation 
\begin{equation*}
    \Phi(t,y',y_{n-1})=(\tau,x',x_{n-1})=(\frac{t-(t_i-t^*_i)}{N/2},\frac{y'-Y'}{N/4}, \frac{y_{n-1}-N}{N/4}).
\end{equation*}
Then,
\begin{equation*}
    \Phi:[t_i-t^*_i,t_i-t^*_i+\frac{N}{2})\times B_{\frac{N}{2}}(Y')\times (N-\frac{N}{2},N+\frac{N}{2})\to [0,1)\times B_1(0)\times (-1,1),
\end{equation*}
and
\begin{equation*}
    \Phi :[t_i-t^*_i,t_i-t^*_i+\frac{N}{4})\times B_{\frac{N}{4}}(Y')\times (N-\frac{N}{4},N+\frac{N}{4})\to [0,\frac{1}{2})\times B_{\frac{1}{2}}(0)\times (-\frac{1}{2},\frac{1}{2}).
\end{equation*}
For $v''_i=v'_i\circ\Phi^{-1}$ and $f''_i=f'_i\circ\Phi^{-1}$, we have
\begin{equation*}
\begin{aligned}
    &\frac{\partial_{\tau\tau}v''_i}{N^2}+\frac{\tilde a_{kl} \partial_{kl}v''_i}{N^2}+\frac{\tilde b_k\partial_{k}v''_i}{N}+\tilde{c}v''_i
    \\
    &\quad=f''_i+\frac{1}{4}n(n+2)\frac{v''_i}{(\frac{N}{4}x_{n-1}+N)^2} \quad \text{in }(0,1)\times B_1(0)\times (-1,1),
    \\
    &v''_i=0 \quad\text{on } \{0\}\times B_1(0)\times (-1,1).
\end{aligned}
\end{equation*}
Multiplying both sides by $N^2$, we have 
\begin{equation*}
    \partial_{\tau\tau}v''_i+\tilde a_{kl} \partial_{kl}v''_i+N\tilde b_k\partial_{k}v''_i+ N^2\tilde{c}v''_i =N^2f''_i+\frac{1}{4}n(n+2)\frac{v''_i}{(1+\frac{x_{n-1}}{4})^2} \quad \text{in }(0,1)\times B_1(0)\times (-1,1).
\end{equation*}
By the $C^1$-estimate up to the boundary, we have, for all $(\tau,x',x_{n-1})\in(0,\frac{1}{2})\times B_{\frac{1}{2}}(0)\times (-\frac{1}{2},\frac{1}{2})$,
\begin{equation*}
\begin{aligned}
    |\nabla v''_i(\tau,x',x_{n-1})|&\leq C\{ \|v''_i\|_{L^\infty((0,1)\times B_{1}(0)\times(-1,1))}
    +\|N^2f''_i\|_{L^\infty((0,1)\times B_{1}(0)\times(-1,1))}
    \\
    &\quad+\|\frac{1}{(1+\frac{x_{n-1}}{4})^2}v''_i\|_{L^\infty((0,1)\times B_{1}(0)\times(-1,1))} \}.
\end{aligned}
\end{equation*}
Since the factor $\frac{N^2}{\rho^2}=(1+\frac{x_{n-1}}{4})^{-2}$ is uniformly bounded from above and below,  we have, for all $(\tau,x',x_{n-1})\in(0,\frac{1}{2})\times B_{\frac{1}{2}}(0)\times (-\frac{1}{2},\frac{1}{2})$,
\begin{equation*}
    |\nabla v''_i(\tau,x',x_{n-1})|\leq C\{ \|v''_i\|_{L^\infty((0,1)\times B_{1}(0)\times(-1,1))}+\|N^2f''_i\|_{L^\infty((0,1)\times B_{1}(0)\times(-1,1))} \}.
\end{equation*}
Scaling back under $\Phi$, we have, for all $(t,y',y_{n-1})\in[t_i-t^*_i,t_i-t^*_i+\frac{N}{4})\times B_{\frac{N}{4}}(Y')\times (N-\frac{N}{4},N+\frac{N}{4})$,
\begin{equation*}
\begin{aligned}
    N|\nabla v'_i(t,y',y_{n-1})|
    &\leq  C\{ \|v'_i\|_{L^\infty([t_i-t^*_i,t_i-t^*_i+\frac{N}{2})\times B_{\frac{N}{2}}(Y')\times (\frac{N}{2},\frac{3N}{2}))}
    \\
    &\quadd+\|N^2f'_i\|_{L^\infty([t_i-t^*_i,t_i-t^*_i+\frac{N}{2})\times B_{\frac{N}{2}}(Y')\times (\frac{N}{2},\frac{3N}{2}))} \}.
\end{aligned}
\end{equation*}
We know that 
\begin{equation*}
    |\tilde{f}_i|\leq CA^{-1}_i\rho^{s-2}e^{\mu(t^*_i-t_i)},
\end{equation*}
and 
\begin{equation*}
    |v'_i|\leq C\rho^{s}e^{\mu(t^*_i-t_i)} .
\end{equation*}
Thus, for all $(t,y',y_{n-1})\in[t_i-t^*_i,t_i-t^*_i+\frac{N}{4})\times B_{\frac{N}{4}}(Y')\times (N-\frac{N}{4},N+\frac{N}{4})$, we have
\begin{equation*}
    |\nabla v'_i(t,y',y_{n-1})|\leq C\rho^{s-1}e^{\mu(t^*_i-t_i)}.
\end{equation*}
In particular,
\begin{equation*}
    |\nabla v'_i(T,Y',N)|\leq Ce^{\mu(t^*_i-t_i)}.
\end{equation*}

Now, if $T\geq t_i-t^*_i+\frac{N}{4}$, we consider $(T-\frac{N}{8}, T+\frac{N}{8})\times B_{\frac{N}{8}}(Y')\times (N-\frac{N}{8},N+\frac{N}{8})$. Note that $T-\frac{N}{8}>t_i-t^*_i$. Consider the transformation 
\begin{equation*}
    \Phi(t,y',y_{n-1})=(\tau,x',x_{n-1})=(\frac{t-T}{N/8},\frac{y'-Y'}{N/8}, \frac{y_{n-1}-N}{N/8}).
\end{equation*}
Then,
\begin{equation*}
    \Phi :(T-\frac{N}{8}, T+\frac{N}{8})\times B_{\frac{N}{8}}(Y')\times (N-\frac{N}{8},N+\frac{N}{8})\to (-1,1)\times B_{1}(0)\times (-1,1),
\end{equation*}
and
\begin{equation*}
    \Phi :(T-\frac{N}{16}, T+\frac{N}{16})\times B_{\frac{N}{16}}(Y')\times (N-\frac{N}{16},N+\frac{N}{16})\to (-\frac{1}{2},\frac{1}{2})\times B_{\frac{1}{2}}(0)\times (-\frac{1}{2},\frac{1}{2}).
\end{equation*}
For $v''_i=v'_i\circ\Phi^{-1}$ and $f''_i=f'_i\circ\Phi^{-1}$, we have
\begin{equation*}
\begin{aligned}
    &\frac{\partial_{\tau\tau}v''_i}{N^2}+\frac{\tilde{a}_{kl}\partial_{kl}v''_i}{N^2}+\frac{\tilde{b}_k\partial_{k}v''_i}{N}+\tilde{c}v''_i
    \\
    &\quad\quadd=f''_i+\frac{1}{4}n(n+2)\frac{v''_i}{(\frac{N}{8}x_{n-1}+N)^2} \quad \text{in }(-1,1)\times B_{1}(0)\times (-1,1).    
\end{aligned}
\end{equation*}
Multiplying both sides by $N^2$, we have
\begin{equation*}
\begin{aligned}
&\partial_{\tau\tau}v''_i+\tilde{a}_{kl}\partial_{kl}v''_i+N\tilde{b}_k\partial_{k}v''_i+\tilde{c}N^2v''_i
\\
&\quadd\quadd=N^2f''_i+\frac{1}{4}n(n+2)\frac{1}{(1+\frac{x_{n-1}}{8})^2}v''_i \quad \text{in }\,(-1,1)\times B_1(0)\times (-1,1).
\end{aligned}
\end{equation*}
By the interior $C^1$-estimate, we have, for all $(\tau,x',x_{n-1})\in(-\frac{1}{2},\frac{1}{2})\times B_{\frac{1}{2}}(0)\times (-\frac{1}{2},\frac{1}{2})$,
\begin{equation*}
\begin{aligned}
    &|\nabla v''_i(\tau,x',x_{n-1})|
\leq C\{ \|v''_i\|_{L^\infty((-1,1)\times B_{1}(0)\times (-1,1))}+\|N^2f''_i\|_{L^\infty((-1,1)\times B_{1}(0)\times (-1,1))}
\\
&\quadd\quadd\quadd\quadd +\|\frac{1}{(1+\frac{x_{n-1}}{8})^2}v''_i\|_{L^\infty((-1,1)\times B_{1}(0)\times(-1,1))} \}.
\end{aligned}
\end{equation*}
As before, we have, for all $(\tau,x',x_{n-1})\in(-\frac{1}{2},\frac{1}{2})\times B_{\frac{1}{2}}(0)\times (-\frac{1}{2},\frac{1}{2})$,
\begin{equation*}
    |\nabla v''_i(\tau,x',x_{n-1})|\leq C\{ \|v''_i\|_{L^\infty((-1,1)\times B_{1}(0)\times(-1,1))}+\|N^2f''_i\|_{L^\infty((-1,1)\times B_{1}(0)\times(-1,1))} \}.
\end{equation*}
Scaling back via $\Phi$, we have, for all $(t,y',y_{n-1})\in(T-\frac{N}{16}, T+\frac{N}{16})\times B_{\frac{N}{16}}(Y')\times (N-\frac{N}{16},N+\frac{N}{16})$,
\begin{equation*}
\begin{aligned}
 N|\nabla v'_i(t,y',y_{n-1})|
&\leq C\{ \|v'_i\|_{L^\infty((T-\frac{N}{8}, T+\frac{N}{8})\times B_{\frac{N}{8}}(Y')\times (N-\frac{N}{8},N+\frac{N}{8}))}
\\
&\quadd+N^2\|f'_i\|_{L^\infty((T-\frac{N}{8}, T+\frac{N}{8})\times B_{\frac{N}{8}}(Y')\times (N-\frac{N}{8},N+\frac{N}{8}))} \}.
\end{aligned}
\end{equation*}
We know that 
\begin{equation*}
    |f'_i|\leq CA^{-1}_i\rho^{s-2}e^{\mu(t^*_i-t_i)},
\end{equation*}
and 
\begin{equation*}
    |v'_i|\leq C\rho^{s}e^{\mu(t^*_i-t_i)}.
\end{equation*}
Thus, we have, for all $(t,y',y_{n-1})\in(T-\frac{N}{8}, T+\frac{N}{8})\times B_{\frac{N}{8}}(Y')\times (N-\frac{N}{8},N+\frac{N}{8})$,
\begin{equation*}
    |\nabla v'_i(t,y',y_{n-1})|\leq C\rho^{s-1}e^{\mu(t^*_i-t_i)}.
\end{equation*}
In particular,
\begin{equation*}
    |\nabla v'_i(T,Y',N)|\leq Ce^{\mu(t^*_i-t_i)}.
\end{equation*}
Therefore, we have 
\begin{equation}\label{eq:gradbound}
    |\nabla \tilde{v}_i|\leq C\rho^{s-1}e^{\mu(t^*_i-t_i)}\quad\text{on } (t_i-t^*_i,t_i-t^*_i+1)\times \Sigma.
\end{equation}
This proves that $t_i-t^*_i$ remains bounded away from zero. Similarly, $T_i-t^*_i$ remains bounded away from zero. Further, we assume 
\begin{equation*}
    \tilde{v}_i\to \hat{v} \quad \text{in every compact subset of }(\tau_-,\tau_+)\times\Sigma.
\end{equation*}
We also have $\tilde{f}_i\to 0$ in every compact subset of $(\tau_-,\tau_+)\times\Sigma$. Additionally, \eqref{eq:mass} and \eqref{eq:gradbound} together with the homogeneous boundary condition on $(t_i,T_i)\times\partial\Sigma$ ensure that the point at which $\rho^{-s}(\cdot)|\tilde{v}_i(0,\cdot)|$ attains its maximum must lie in the interior of $\Sigma$. Consequently, $\hat{v}\neq0$,
\begin{equation}\label{eq:vhatbound}
    e^{\mu t}|\hat{v}(t,\theta)|\leq C \quad\text{for any }(t,\theta)\in(\tau_-,\tau_+)\times\Sigma,
\end{equation}
and
\begin{equation*}
    \mathcal{L}\hat{v}=0 \quad \text{on }(\tau_-,\tau_+)\times\Sigma.
\end{equation*}
Moreover,
\begin{equation}\label{eq:vhatlimit}
    \lim\limits_{t\to\tau_*}\hat{v}(t,\cdot)=0,
\end{equation}
where $\tau_*=\tau_-$ or $\tau_+$ if it is finite. 

Next, we proceed as in the proof of Lemma \ref{lem:Uniqueness}. For any $i\geq 1$, set
\begin{equation*}
    \hat{v}_i(t)=\int_{\Sigma} \hat{v}(t,\theta)\phi_i(\theta)\,d\theta.
\end{equation*}
Then, $L_i\hat{v}_i=0$, and hence $\hat{v}_i$ is a linear combination of the basis of $\text{Ker}(L_i)$. By the assumption, $\hat{v}_i=0$ for any $i$ with $\gamma_i<\mu$. We now take an $i$ with $\gamma_i>\mu$. Then,
\begin{equation*}
    \hat{v}_i(t)=c_1e^{-\gamma_it}+c_2e^{\gamma_it},
\end{equation*}
where $c_1,c_2\in\mathbb{R}$. By \eqref{eq:vhatbound}, for any $t\in(\tau_-,\tau_+)$,
\begin{equation*}
    |e^{\mu t}\hat{v}_i(t)|\leq C.
\end{equation*}
If $\tau_+=\infty$, then $c_2=0$ and hence $\hat{v}_i(t)=c_1e^{-\gamma_it}$, which decays exponentially as $t\to\infty$. If $\tau_+$ is finite, then $\lim\limits_{t\to\tau_+}\hat{v}_i(t)=0$ by \eqref{eq:vhatlimit}. Similarly, if $\tau_-=-\infty$ then $c_1=0$, and hence $\hat{v}_i(t)=c_2e^{\gamma_i t}$, which decays exponentially as $t\to -\infty$. If $\tau_-$ is finite, then $\lim\limits_{t\to\tau_-}\hat{v}_i (t)=0$ by \eqref{eq:vhatlimit}. There are no boundary terms because of the exponential decay and vanishing boundary values, so by integration by parts on $\hat{v}_iL_i\hat{v}_i$ from $\tau_-$ to $\tau_+$, we get
\begin{equation*}
    \int^{\tau_+}_{\tau_-}\bigl[(\partial_t\hat{v}_i)^2 + (\lambda_i+\beta^2)\hat{v}_i^2 \bigr]\,dt=0.
\end{equation*}
In conclusion, $\hat{v}_i=0$ for any $i$ and hence $\hat{v}=0$, which leads to a contradiction.
\end{proof}

To construct solutions of \eqref{eq:OGLINEAREQ}, we first consider exponentially decaying solutions in appropriate finite-dimensional subspaces of $L^2(\Sigma)$.
\begin{lemma}\label{lem:FINITESUB} Let $\alpha\in (0,1)$, $s>2$, $\mu>\gamma_I$ for some $I\geq 2$, and $f\in \Lambda^{0,\alpha}_{\mu,s-2}([t_0,\infty)\times\Sigma)$ with $f(t,\cdot)\in\operatorname{span}\{\phi_1,\phi_2,\cdots,\phi_I\}$ for any $t\geq t_0$. Then, there exists a unique solution $v\in \Lambda^{2,\alpha}_{\mu,s}([t_0,\infty)\times\Sigma)$ of \eqref{eq:OGLINEAREQ} with $v(t,\cdot)\in \operatorname{span}\{\phi_1,\phi_2,\cdots,\phi_I\}$ for any $t\geq t_0$. Moreover, the correspondence $f\mapsto v$ is linear, and 
\begin{equation}\label{eq:FINITESUBEST}
    {\lVert v\rVert}_{\Lambda_{\mu,s}^{2,\alpha}([t_0,\infty)\times\Sigma)} \leq C{\lVert f\rVert}_{\Lambda_{\mu,s-2}^{0,\alpha}([t_0,\infty)\times\Sigma)},
\end{equation} 
where $C$ is a positive constant depending only on $n,\alpha,\mu,\Sigma$, and $I$, and is independent of $t_0$.
\end{lemma}
\begin{proof}

Write $C_I=\max_{1\leq i \leq I}C_i$ as in Theorem \ref{thm:weightedEV}. Throughout the proof, write $A \lesssim_I B$ to mean $A \le C_I B$,
where $C_I$ depends only on $n,\alpha,\mu,\Sigma$, and $\max_{1\le i\le I}\lambda_i$,
and is independent of $t_0$.
For each $i=1,2,\cdots,I$, set 
\begin{equation*}
    f_i(t)=\int_\Sigma f(t,\theta)\phi_i(\theta)\,d\theta,
\end{equation*}
 and compute
\begin{equation*}
\begin{aligned}
    |f_i(t)|&\leq \int_\Sigma |f(t,\theta)||\phi_i(\theta)|\,d\theta
    \\
    &\lesssim_I{\lVert f\rVert}_{\Lambda_{\mu,s-2}^{0}([t_0,\infty)\times\Sigma)}e^{-\mu t}\int_\Sigma \rho^{2s-2}(\theta)\,d\theta \lesssim_I e^{-\mu t}{\lVert f\rVert}_{\Lambda_{\mu,s-2}^{0}([t_0,\infty)\times\Sigma)}.
\end{aligned}
\end{equation*}
Hence,
\begin{equation*}
    e^{\mu t}|f_i(t)|\lesssim_I{\lVert f\rVert}_{\Lambda_{\mu,s-2}^{0}([t_0,\infty)\times\Sigma)}.
\end{equation*}
Similarly, for $t\geq t_0+1$ and $t'\in[t-1,t+1]$ such that $t'\neq t$, we have
\begin{equation*}
    e^{\mu t}|f_i(t)-f_i(t')|\leq e^{\mu t}\int_\Sigma |f(t,\theta)-f(t',\theta)||\phi_i(\theta)|\,d\theta,
\end{equation*}
and 
\begin{equation*}
    e^{\mu t}|f(t,\theta)-f(t',\theta)|\leq C{\lVert f\rVert}_{\Lambda_{\mu,s-2}^{0,\alpha}([t_0,\infty)\times\Sigma)}\,\rho(\theta)^{s-2}|t-t'|^\alpha.
\end{equation*}
Then,
\begin{equation*}
\begin{aligned}
  e^{\mu t}|f_i(t)-f_i(t')|&\leq e^{\mu t}\int_\Sigma |f(t,\theta)-f(t',\theta)||\phi_i(\theta)|\,d\theta    
\\
    &\leq C{\lVert f\rVert}_{\Lambda_{\mu,s-2}^{0,\alpha}([t_0,\infty)\times\Sigma)}|t-t'|^\alpha \int_\Sigma \rho(\theta)^{s-2}|\phi_i(\theta)|\,d\theta
\\
  &\lesssim_I {\lVert f\rVert}_{\Lambda_{\mu,s-2}^{0,\alpha}([t_0,\infty)\times\Sigma)}|t-t'|^\alpha \int_\Sigma \rho(\theta)^{2s-2}\,d\theta.
\end{aligned}
\end{equation*}
Thus,
\begin{equation*}
    {\lVert f_i\rVert}_{C_{\mu}^{0,\alpha}([t_0,\infty))}\lesssim_I {\lVert f\rVert}_{\Lambda_{\mu,s-2}^{0,\alpha}([t_0,\infty)\times\Sigma)},
\end{equation*}
and 
\begin{equation*}
    f(t,\theta)=\sum_{i=1}^I f_i(t)\phi_i(\theta).
\end{equation*}
Let $L_i$ be the linear operator as in \eqref{eq:defLi}. We first consider the ordinary differential equation 
\begin{equation}\label{eq:Livi}
    L_iv_i=f_i.
\end{equation}
We claim that there exists a solution $v_i\in C^{2,\alpha}_\mu $  of \eqref{eq:Livi} satisfying
\begin{equation}\label{eq:vifiestimate}
    {\lVert v_i\rVert}_{C_{\mu}^{2,\alpha}([t_0,\infty))}\lesssim_I{\lVert f_i\rVert}_{C_{\mu}^{0,\alpha}([t_0,\infty))}.
\end{equation}
For all $i\geq1$, we have that $\text{Ker}(L_i)$ is spanned by $e^{\gamma_i t}$ and $e^{-\gamma_i t}$. Set
\begin{equation*}
    v_i(t)=\frac{e^{-\gamma_i t}}{2\gamma_i}\int_t^{\infty}e^{\gamma_i s}f_i(s)\,ds-\frac{e^{\gamma_i t}}{2\gamma_i}\int_t^{\infty} e^{-\gamma_i s}f_i(s)\,ds.
\end{equation*}
A simple computation yields, for $t\geq t_0$,
\begin{equation*}
    e^{\mu t}|v_i(t)|\lesssim_I\sup\limits_{t\geq t_0}e^{\mu t}|f_i(t)|\lesssim_I{\lVert f_i\rVert}_{C_{\mu}^{0}([t_0,\infty))}.
\end{equation*}
By a straightforward computation, we have
\begin{equation*}
    v_i'(t)=-\frac{1}{2}\Big[e^{-\gamma_i t}\int_t^\infty e^{\gamma_i s}f_i(s)\,ds + e^{\gamma_i t}\int_t^\infty e^{-\gamma_i s}f_i(s)\,ds \Big],
\end{equation*}
and 
\begin{equation*}
    v''_i(t)=\frac{\gamma_i e^{-\gamma_i t}}{2}\int_t^\infty e^{\gamma_i s}f_i(s)\,ds-\frac{\gamma_i e^{\gamma_i t}}{2}\int_t^\infty e^{-\gamma_i s}f_i(s)\,ds +f_i(t).
\end{equation*}
We obtain, for $t\geq t_0$, 
\begin{equation*}
    e^{\mu t}|v_i'(t)|+e^{\mu t}|v_i''(t)| \lesssim_I {\lVert f_i\rVert}_{C_{\mu}^{0}([t_0,\infty))}.
\end{equation*}
For the Hölder semi-norms of $v_i''$, we write
\begin{equation*}
    v_i''=P+f_i,
\end{equation*}
where
\begin{equation*}
    P(t)=\frac{\gamma_i e^{-\gamma_i t}}{2}\int_t^\infty e^{\gamma_i s}f_i(s)\,ds-\frac{\gamma_i e^{\gamma_i t}}{2}\int_t^\infty e^{-\gamma_i s}f_i(s)\,ds.
\end{equation*}
Then,
\begin{equation*}
    P'(t)=-\frac{\gamma_i^2}{2}\Biggl(e^{-\gamma_i t}\int_t^\infty e^{\gamma_i s}f_i(s)\,ds-\int_t^\infty e^{-\gamma_i s}f_i(s)\,ds\Biggr).
\end{equation*}
Similarly, for $t\geq t_0$, we have
\begin{equation*}
    e^{\mu t}|P'(t)|\lesssim_I {\lVert f_i\rVert}_{C_{\mu}^{0}([t_0,\infty))},
\end{equation*}
and hence, for $t\geq t_0+1$,
\begin{equation*}
    e^{\mu t}[P]_{C^\alpha([t-1,t+1])} \lesssim_I{\lVert f_i\rVert}_{C_{\mu}^{0}([t_0,\infty))}.
\end{equation*}
Therefore, for $t\geq t_0+1$,
\begin{equation*}
    e^{\mu t}[v_i'']_{C^\alpha([t-1,t+1])} \lesssim_I{\lVert f_i\rVert}_{C_{\mu}^{0,\alpha}([t_0,\infty))}.
\end{equation*} Combining these, we have \eqref{eq:vifiestimate}.
With the solution $v_i$ of \eqref{eq:Livi} for $i=1,2,\cdots,I$, we set
\begin{equation*}
    v(t,\theta)=\sum\limits_{i=1}^I v_i(t)\,\phi_i(\theta).
\end{equation*}
Then, $\mathcal{L}v=f$ and, by \eqref{eq:vifiestimate}
\begin{equation*}
{\lVert v\rVert}_{\Lambda_{\mu,s}^{2,\alpha}([t_0,\infty)\times\Sigma)} \lesssim_I\sum\limits_{i=1}^I  {\lVert v_i\rVert}_{C_{\mu}^{2,\alpha}([t_0,\infty))}
\lesssim_I\sum\limits_{i=1}^I  {\lVert f_i\rVert}_{C_{\mu}^{0,\alpha}([t_0,\infty))}\lesssim_I{\lVert f\rVert}_{\Lambda_{\mu,s-2}^{0,\alpha}([t_0,\infty)\times\Sigma)}.
\end{equation*}
Therefore, $v$ is the desired solution. It is easy to see that such a $v$ is unique.
\end{proof}
It is important to note that we only have uniqueness for solutions $w$ such that $w(t,\cdot)\in \operatorname{span}\{\phi_1,\phi_2,\cdots,\phi_I\}$ for $t\geq t_0$. We proceed by constructing solutions in infinite-dimensional subspaces of $L^2(\Sigma)$ with homogeneous boundary conditions.

\begin{lemma}\label{lem:INFINITESUB} Let $\alpha\in(0,1)$, $s>2$, $\mu>\gamma_1$, and $\mu\neq\gamma_i$ for any $i\geq 2$, and $f\in \Lambda^{0,\alpha}_{\mu,s-2}([t_0,\infty)\times\Sigma)$ with $\int_\Sigma f(t,\cdot)\phi_i(\theta)\,d\theta=0$ for any $i$ with $\gamma_i<\mu$ and any $t\geq t_0$. Then, there exists a unique solution $v\in \Lambda^{2,\alpha}_{\mu,s}([t_0,\infty)\times\Sigma)$ of $\mathcal{L}v=f$ with $v=0$ on $\{t_0\}\times \Sigma$. Moreover,
\begin{equation}\label{eq:weaklowersemi}
    {\lVert v\rVert}_{\Lambda_{\mu,s}^{2,\alpha}([t_0,\infty)\times\Sigma)} \leq C{\lVert f\rVert}_{\Lambda_{\mu,s-2}^{0,\alpha}([t_0,\infty)\times\Sigma)}, 
\end{equation}
where $C$ is a positive constant depending only on $n,\alpha,s,\mu,\Sigma$ but independent of $t_0$.
\end{lemma}

\begin{proof} Take any $T\geq t_0+4$. We first prove there exists a solution $v_T\in C^{2,\alpha}([t_0,T]\times\Sigma)$ of
\begin{equation}\label{eq:LvT}
\begin{aligned}
\mathcal{L}v_T&=f\quad\text{in }(t_0,T)\times\Sigma,
\\
v_T&=0\quad\text{on }((\{t_0\}\cup\{T\})\times\Sigma)\cup  ((t_0,T)\times\partial\Sigma).
\end{aligned}
\end{equation}
Consider the energy functional
\begin{equation*}
    \mathcal{E}_T(v)=\int_{t_0}^T \int_\Sigma \Big[(\partial_t v)^2+|\nabla_\theta v|^2 + \frac{(n-2)^2}{4}v^2+\frac{n(n+2)}{4}\rho^{-2}v^2+2fv \Big]\,dt\,d\theta.
\end{equation*}
Set
\begin{equation*}
    \mathcal{X} =\{u \in H^1(\Sigma); \int_\Sigma u(\theta)\phi_i(\theta)\,d\theta=0\quad \text{for any } i\text{ with }\gamma_i<\mu \}.
\end{equation*}
Then, for any $u\in\mathcal{X}$,
\begin{equation*}
    \int_\Sigma |\nabla_\theta u|^2 +\frac{n(n+2)}{4}\rho^{-2}u^2\,d\theta \geq \lambda_1 \int_\Sigma u^2\, d\theta.
\end{equation*}
Hence, for any $v\in H^1_0((t_0,T)\times\Sigma)$ with $v(t,\cdot)\in\mathcal{X}$ for any $t\in(t_0,T)$, we have
\begin{equation*}
    \mathcal{E}_T(v)\geq\int_{t_0}^T \int_\Sigma \Big[(\partial_t v)^2 + \big(\beta^2+\lambda_1\big)v^2+2fv \Big]\,dt\,d\theta.
\end{equation*}
We conclude that $\mathcal{E}_T$ is coercive and weakly lower semi-continuous. Hence, we can find a minimizer $v_T$ of $\mathcal{E}_T$ in the space 
\begin{equation*}
    \{ v\in H^1_0((t_0,T)\times\Sigma);\,v(t,\cdot)\in\mathcal{X} \ \text{for any }t\in(t_0,T) \}.
\end{equation*}
Since $f(t,\cdot)\in\mathcal{X}$ for any $t\in(t_0,T)$, $v_T$ is a solution of \eqref{eq:LvT}, with $v_T(t,\cdot)\in\mathcal{X}$ for any $t\in (t_0,T).$
\par\indent By Lemma \ref{lem:Linfinity}, we have 
\begin{equation*}
    \sup\limits_{(t,\theta)\in[t_0,T]\times\Sigma}\rho^{-s}(\theta) e^{\mu t}|v_T(t,\theta)|\leq C\sup\limits_{(t,\theta)\in[t_0,T]\times\Sigma} \rho^{-s+2}(\theta) e^{\mu t}|f(t,\theta)|,
\end{equation*}
where $C$ is a positive constant depending only on $n,\alpha,\mu$, and $\Sigma$, independent of $t_0$. For each fixed $T_0>t_0$, consider $[t_0,T_0]\times\Sigma \subset [t_0,T_0+1]\times\Sigma$. By the interior and boundary Schauder estimates, $v_T(t_0,\theta)=0$, and passing to a subsequence, $v_T$ converges to a $C^{2,\alpha}$ solution of $\mathcal{L}v=f$ in $[t_0,T_0]\times\Sigma$ with $v=0$ on $\{t_0\}\times\Sigma$, as $T\to\infty$. By a diagonalization process, $v_T$ converges to a $C^{2,\alpha}$ solution $v$ of $\mathcal{L}v=f$ in $[t_0,\infty)\times\Sigma$, with $v=0$ on $\{t_0\}\times\Sigma$. Moreover,
\begin{equation*}
\sup\limits_{(t,\theta)\in[t_0,\infty)\times\Sigma}\rho^{-s} e^{\mu t}|v(t,\theta)|\leq C\sup\limits_{(t,\theta)\in[t_0,\infty)\times\Sigma} \rho^{-s+2} e^{\mu t}|f(t,\theta)|,
\end{equation*}
or
\begin{equation}\label{eq:SUBEST}
{\lVert v\rVert}_{\Lambda_{\mu,s}^{0}([t_0,\infty)\times\Sigma)} \leq C{\lVert f\rVert}_{\Lambda_{\mu,s-2}^{0,\alpha}([t_0,\infty)\times\Sigma)}, 
\end{equation}
where  $C$ is a positive constant depending only on $n,\alpha,\mu$, and $\Sigma$, and is independent of $t_0$. By substituting \eqref{eq:SUBEST} in \eqref{eq:Apriori} with $\varphi=0$, we have \eqref{eq:weaklowersemi}.
\end{proof}

Finally, we establish the invertibility of $\mathcal{L}$.
\begin{theorem}  
Let $\alpha\in(0,1)$, $s>2$, $\mu>\gamma_1$ and $\mu\neq\gamma_i$ for any $i$, and $f\in \Lambda^{0,\alpha}_{\mu,s-2}([t_0,\infty)\times\Sigma)$. Then $\mathcal{L}v=f$ admits a solution $v\in\Lambda^{2,\alpha}_{\mu,s}([t_0,\infty)\times\Sigma)$ and 
\begin{equation}
{\lVert v\rVert}_{\Lambda_{\mu,s}^{2,\alpha}([t_0,\infty)\times\Sigma)} \leq C{\lVert f\rVert}_{\Lambda_{\mu,s-2}^{0,\alpha}([t_0,\infty)\times\Sigma)},
\end{equation}
where $C$ is a positive constant depending only on $n,\alpha,\mu$, and $\Sigma$, and is independent of $t_0$. Moreover, the correspondence $f\mapsto v$ is linear.
\end{theorem}
\begin{proof}
    Take $I$ to be the largest integer such that $\gamma_I<\mu$. Set, for $i=1,2,\cdots,I$,
    \begin{equation*}
        f_i(t)=\int_\Sigma f(t,\theta)\phi_i(\theta)\,d\theta.
    \end{equation*}
    First, let $v_1\in\Lambda^{2,\alpha}_{\mu,s}([t_0,\infty)\times\Sigma)$ be the unique solution of 
    \begin{equation*}
        \mathcal{L}v_1=\sum_{i=1}^If_i\phi_i \quad\text{in }[t_0,\infty)\times\Sigma,
    \end{equation*}
    as in Lemma \ref{lem:FINITESUB}. By Lemma \ref{lem:INFINITESUB}, let $v_2\in \Lambda^{2,\alpha}_{\mu,s}([t_0,\infty)\times\Sigma)$ be the unique solution of 
    \begin{equation*}
    \begin{aligned}
        \mathcal{L}v_2&=f-\sum_{i=1}^If_i\phi_i \quad \text{in } [t_0,\infty)\times\Sigma,
        \\
        v_2&=0 \quad \text{on } \{t_0\}\times\Sigma.
        \end{aligned}
   \end{equation*}
    Then, $v=v_1+v_2$ is the desired solution.
\end{proof}
\section{Convergence of Approximate Solutions via Contraction}\label{sec-APPROX}
In this section, we define approximate solutions and prove Theorem 1.1 using the contraction mapping theorem.
Recall
\begin{equation}\tag{$\star$}
  v_{tt}
  + \Delta_\theta v
  - \tfrac14\,(n-2)^2\,v
  - \tfrac14\,n(n-2)\,v^{\frac{n+2}{n-2}}
  = 0.
\end{equation}
This is simply \eqref{eq:CYLINDRICALLN}, expressed in terms of $v$. Throughout, set
\begin{equation}\label{eq:NV}
\mathcal{N}(v)=v_{tt}+\Delta_\theta v-\frac{1}{4}(n-2)^2v-\frac{1}{4}n(n-2)v^{\frac{n+2}{n-2}}.
\end{equation}
Then, $v$ is a solution of $(\star)$ if $\mathcal{N}(v)=0$.
We proceed to the main theorem of this section.
\begin{theorem}\label{thm:CONTRACTION} 
Let $\xi$ be the positive solution of \eqref{eq:xi}-\eqref{eq:xiB}, $\mathcal{I}$ the index set associated with $\xi$, and $\mu>\gamma_1$ with $\mu\notin\mathcal{I}$. Suppose that $\hat{v}\in C^{2,\alpha}([0,\infty)\times\Sigma)$ satisfies 
\begin{equation}\label{eq:ASS1}
    |\hat{v}-\xi|+\rho|\nabla(\hat{v}-\xi)|\leq \rho^s \epsilon(t),
\end{equation}
where $\epsilon(t)$ is a decreasing function such that $\epsilon(t) \to 0$ as $t\to\infty$, and, for any $(t,\theta)\in[0,\infty)\times\Sigma$,
\begin{equation}\label{eq:ASS2}
    |\mathcal{N}(\hat{v})(t,\theta)|+\rho(\theta)|\nabla(\mathcal{N}(\hat{v}))(t,\theta)|\leq K\rho(\theta)^{s-2}e^{-\mu t},
\end{equation}
for some positive constant $K$. Then, there exists a $t_0>0$ and a solution $v\in C^{2,\alpha}([t_0,\infty)\times\Sigma)$ of $(\star)$ such that, for any $(t,\theta)\in(t_0,\infty)\times\Sigma$,
\begin{equation*}
    |v(t,\theta)-\hat{v}(t,\theta)|\leq C\rho(\theta)^{s}e^{-\mu t},
\end{equation*}
where $C$ is a positive constant.
\end{theorem}
\begin{proof}
    Given approximate solution $\hat{v}$, we will find a $w\in \Lambda_{\mu,s}^{2,\alpha}([t_0,\infty)\times\Sigma)$ such that 
    \begin{equation*}
        \mathcal{N}(\hat{v}+w)=0. 
    \end{equation*}
    
    {\it Step 1.} Write
   \begin{equation*}
        0=\mathcal{N}(\hat{v}+w)=\mathcal{N}(\hat{v})+\mathcal{L}w-P(w),
    \end{equation*}
    where
    \begin{equation*}
        P(w)=\frac{1}{4}n(n-2)[(\hat{v}+w)^{\frac{n+2}{n-2}}-\hat{v}^{\frac{n+2}{n-2}}]-\frac{1}{4}n(n+2)\frac{w}{\rho^2}.
    \end{equation*}
    Then, 
    \begin{equation*}
        \mathcal{L}w=P(w)-\mathcal{N}(\hat{v}),
    \end{equation*}
    and with the operator $\mathcal{L}^{-1}$, we can rewrite it further as 
    \begin{equation*}
        w=\mathcal{L}^{-1}[P(w)-\mathcal{N}(\hat{v})].
    \end{equation*}
    We define the mapping $\mathcal{T}$ by
    \begin{equation*}
        \mathcal{T}(w)=\mathcal{L}^{-1}[P(w)-\mathcal{N}(\hat{v})].
    \end{equation*}
    We will prove that $\mathcal{T}$ is a contraction on some ball in $\Lambda_{\mu,s}^{2,\alpha}([t_0,\infty)\times\Sigma)$, for some $t_0$ large.
    For convenience, set
    \begin{equation*}
        \mathcal{X}_{B,t_0}=\{w\in\Lambda_{\mu,s}^{2,\alpha}([t_0,\infty)\times\Sigma); {\lVert w\rVert}_{\Lambda_{\mu,s}^{2,\alpha}([t_0,\infty)\times\Sigma)}\leq B\}.
    \end{equation*}
    
    {\it Step 2.} We show $\mathcal{T}$ maps $\mathcal{X}_{B,t_0}$ to itself, for some fixed $B$ and any $t_0$ sufficiently large; namely, for any $w\in\Lambda_{\mu,s}^{2,\alpha}([t_0,\infty)\times\Sigma)$ with ${\lVert w\rVert}_{\Lambda_{\mu,s}^{2,\alpha}([t_0,\infty)\times\Sigma)}\leq B$, we have $\mathcal{T}(w)\in\Lambda_{\mu,s}^{2,\alpha}([t_0,\infty)\times\Sigma)$ and ${\lVert \mathcal{T}(w)\rVert}_{\Lambda_{\mu,s}^{2,\alpha}([t_0,\infty)\times\Sigma)}\leq B$. 
    %%%
    Note that by \eqref{eq:ASS2}, we have 
    \begin{equation*}
        {\lVert \mathcal{N}(\hat{v})\rVert}_{\Lambda_{\mu,s-2}^{1}([t_0,\infty)\times\Sigma)}\leq K.
    \end{equation*}
    Next, set 
    \begin{equation}\label{eq:Q}
        Q(w)=\frac{1}{4}n(n+2)\int^1_0 \Big[(\hat{v}+\tau w)^{\frac{4}{n-2}}-\xi^{\frac{4}{n-2}}\Big]\,d\tau.
    \end{equation}
    Then, $P(w)=wQ(w)$. By \eqref{eq:ASS1} and writing $s-\beta=2$ and $\xi=\rho^{-\beta}$, we have
    \begin{equation*}
        1-\epsilon(t)\rho^2 \leq\frac{\hat{v}}{\xi}\leq 1+\epsilon(t)\rho^2,
    \end{equation*}
    and
    \begin{equation*}
        \frac{|w|}{\xi}\leq  B e^{-\mu t}\rho^{2\beta+2}.
    \end{equation*}
    Then, for some fixed $\delta\in(0,\frac{1}{2}]$, we can choose a sufficiently large $t_0>0$ so that, for all $t\geq t_0$ and $\tau\in (0,1),$
    \begin{equation*}
        \Big|\frac{\hat{v}+\tau w}{\xi}-1\Big|\leq (\epsilon(t)+Be^{-\mu t})\rho^2<\delta. 
    \end{equation*}
    Writing $f(t)=t^{\frac{4}{n-2}},\,f'(t)=\frac{4}{n-2}t^{\frac{6-n}{n-2}},\, f''(t)=\frac{4(6-n)}{(n-2)^2}t^{\frac{8-2n}{n-2}}$, and $\rho=\xi^{-\frac{2}{n-2}}$, we have
\begin{equation}
\begin{aligned}
 \rho^2|Q(w)|&\leq \frac{1}{4}n(n+2)\int^1_0 \Big|\Big(\frac{\hat{v}+\tau w}{\xi}\Big)^{\frac{4}{n-2}}-1\Big|\, d\tau=\frac{1}{4}n(n+2)\int^1_0 \Big|f\Big(\frac{\hat{v}+\tau w}{\xi}\Big)-f(1)\Big|\, d\tau    
\\
    &\leq \frac{1}{4}n(n+2)\|f'\|_{L^\infty([1-\delta,1+\delta])}\int_0^1 \Big|\frac{\hat{v}+\tau w}{\xi}-1\Big|\, d\tau \leq C(\epsilon(t)+Be^{-\mu t})\rho^2.
\end{aligned}
\end{equation}
    Thus,
    \begin{equation*}
        |Q(w)|\leq \epsilon(t)+Be^{-\mu t}.
    \end{equation*}
    Next,
\begin{equation}\label{eq:equation1}
\begin{aligned}
\rho^3\big|\nabla Q(w)\big|&\leq\frac{1}{4}n(n+2)\rho^3\int^1_0 \big|\nabla f(\hat{v}+\tau w)-\nabla f(\xi)\big|\, d\tau    
\\
&=\rho^3\int^1_0 \big|f'(\hat{v}+\tau w)\tau \nabla w+f'(\hat{v}+\tau w)\nabla \hat{v}-f'(\xi)\nabla\xi\big| \, d\tau
\\
&=\rho^3\int^1_0 \Big|f'(\hat{v}+\tau w)\tau\nabla w+f'(\hat{v}+\tau w)\nabla \hat{v}
\\
&\quadd\quad\quad+\big[f'(\hat{v}+\tau w)\nabla\xi-f'(\hat{v}+\tau w)\nabla\xi\big] -f'(\xi)\nabla\xi\Big| \, d\tau
\\
&\leq\int^1_0 \Big| \rho^3f'(\hat{v}+\tau w)\big[\nabla(\hat{v}-\xi)+\tau\nabla w\big]\Big|\, d\tau+\int_0^1\Big|\rho^3\big[f'(\hat{v}+\tau w)-f'(\xi)\big]\nabla\xi\Big|\, d\tau.
\end{aligned}
\end{equation}
Using $\rho^{\frac{n}{2}}\xi^{\frac{n-6}{n-2}}=\rho^3$, we have
    \begin{equation*}
        \rho^3f'(\hat{v}+\tau w)=\rho^{\frac{n}{2}}\xi^{\frac{n-6}{n-2}}\frac{4}{n-2}\Big(\hat{v}+\tau w\Big)^{\frac{6-n}{n-2}}=\rho^{s-1}f'\Big(\frac{\hat{v}+\tau w}{\xi}\Big),
    \end{equation*}
    and similarly,
    \begin{equation*}
        \rho^3f'(\xi)=\rho^{\frac{n}{2}}\xi^{\frac{n-6}{n-2}}\frac{4}{n-2}\xi^{\frac{6-n}{n-2}}=\frac{4}{n-2}\rho^{\frac{n}{2}}=\rho^{s-1}f'(1).
    \end{equation*}
    Then, estimating the final term in \eqref{eq:equation1} and using $s-1=\frac{n}{2}$, we have
    \begin{equation}\label{eq:equation2}
    \begin{aligned}
        \int^1_0 \Big|\rho^3f'(\hat{v}+\tau w)\big[\nabla(\hat{v}-\xi)+\tau\nabla w\big]\Big|\, d\tau &= \rho^{\frac{n}{2}}\int^1_0 \Big|f'\Big(\frac{\hat{v}+\tau w}{\xi}\Big)\big[\nabla(\hat{v}-\xi)+\tau\nabla w\big]\Big|\,d\tau
        \\
        &\leq\rho^{\frac{n}{2}}\|f'\|_{L^\infty([1-\delta,1+\delta])}(\epsilon(t)+Be^{-\mu t})\rho^{s-1}
        \\
        &=\|f'\|_{L^\infty([1-\delta,1+\delta])}(\epsilon(t)+Be^{-\mu t})\rho^n.
    \end{aligned}
    \end{equation}
    Similarly, by Lemma \ref{lem:gradxi}, we get
    \begin{equation*}
        \big|\nabla\xi\big|\leq C\rho^{-s+1}\quad\text{on } \Sigma.
    \end{equation*}
Estimating the penultimate term in \eqref{eq:equation1}, we get
\begin{equation}\label{eq:equation3}
\begin{aligned}
  \int_0^1\Big|\rho^3\big[f'(\hat{v}+\tau w)-f'(\xi)\big]\nabla\xi\Big|\, d\tau&=\int_0^1\Big|\rho^{s-1}\big[f'\Big(\frac{\hat{v}+\tau w}{\xi}\Big)-f'(1)\big]\nabla\xi\Big|\, d\tau    
\\
    &\leq \rho^{s-1}\|f''\|_{L^\infty([1-\delta,1+\delta])}\big|\nabla\xi\big|\int_0^1\Big|\frac{\hat{v}+\tau w}{\xi}-1\Big| \, d\tau
\\
  &\leq C\|f''\|_{L^\infty([1-\delta,1+\delta])}\int_0^1\Big|\frac{\hat{v}+\tau w}{\xi}-1\Big|\, d\tau
\\
  &\leq C\|f''\|_{L^\infty([1-\delta,1+\delta])}(\epsilon(t)+Be^{-\mu t})\rho^2.
\end{aligned}
\end{equation}
Combining \eqref{eq:equation1}, \eqref{eq:equation2}, and \eqref{eq:equation3}, we obtain
    \begin{equation*}
        \rho\big|\nabla Q(w)\big|\leq C(\epsilon(t)+Be^{-\mu t}),
    \end{equation*}
    and thus,
    \begin{equation*}
        |Q(w)|+\rho|\nabla Q(w)|\leq C'(\epsilon(t)+Be^{-\mu t}),
    \end{equation*}
    where $C$ is independent of $t_0$. In particular,
    \begin{equation*}
        Q(w)\in \Lambda^1_{0,0}([t_0,\infty)\times\Sigma), 
    \end{equation*}
    with
    \begin{equation*}
        {\lVert Q(w)\rVert}_{\Lambda_{0,0}^{1}([t_0,\infty)\times\Sigma)}\leq C'(\epsilon(t_0)+Be^{-\mu t_0}).
    \end{equation*}
    By Theorem 2.9, we have
\begin{equation}
\begin{aligned}
    {\lVert \mathcal{T}(w)\rVert}_{\Lambda_{\mu,s}^{2,\alpha}([t_0,\infty)\times\Sigma)}&\leq C {\lVert P(w)-\mathcal{N}(\hat{v})\rVert}_{\Lambda_{\mu,s-2}^{0,\alpha}([t_0,\infty)\times\Sigma)}
\\
    &\leq C \{{\lVert P(w)\rVert}_{\Lambda_{\mu,s-2}^{1}([t_0,\infty)\times\Sigma)}+{\lVert \mathcal{N}(\hat{v})\rVert}_{\Lambda_{\mu,s-2}^{1}([t_0,\infty)\times\Sigma)}\}
\\
    &\leq C {\lVert w\rVert}_{\Lambda_{\mu,s}^{1}([t_0,\infty)\times\Sigma)}{\lVert Q(w)\rVert}_{\Lambda_{0,0}^{1}([t_0,\infty)\times\Sigma)}+CK
\\
    &\leq C \{C'(\epsilon(t_0)+Be^{-\mu t_0})B+K\},
\end{aligned}
\end{equation}
where $C$ is independent of $t_0$. We first take $B\geq 2CK$, and then $t_0$ large so that 
\begin{equation*}
CC'(\epsilon(t_0)+Be^{-\mu t_0})\leq 1/2.
\end{equation*}
Then,
\begin{equation*}
        {\lVert \mathcal{T}(w)\rVert}_{\Lambda_{\mu,s}^{2,\alpha}([t_0,\infty)\times\Sigma)} \leq B.
\end{equation*}
This is the desired estimate.
\par\indent{\it Step 3.} We prove that $\mathcal{T}:\mathcal{X}_{B,t_0}\to \mathcal{X}_{B,t_0}$ is a contraction, i.e., for any $w_1,w_2\in \mathcal{X}_{B,t_0}$,
\begin{equation*}
    {\lVert \mathcal{T}(w_1)-\mathcal{T}(w_2)\rVert}_{\Lambda_{\mu,s}^{2,\alpha}([t_0,\infty)\times\Sigma)}\leq \lambda{\lVert w_1-w_2\rVert}_{\Lambda_{\mu,s}^{2,\alpha}([t_0,\infty)\times\Sigma)},
\end{equation*}
for some constant $\lambda\in(0,1)$. We note
\begin{equation*}
        \mathcal{T}(w_1)-\mathcal{T}(w_2)=\mathcal{L}^{-1}(P(w_1)-P(w_2)),
\end{equation*}
and
\begin{equation*}
        P(w_1)-P(w_2)=w_1Q(w_1)-w_2Q(w_2)=(w_1-w_2)Q(w_1)+w_2(Q(w_1)-Q(w_2)).
\end{equation*}
By \eqref{eq:Q} and using $\rho^2=\xi^{-\frac{4}{n-2}}$, we have
\begin{equation*}
\begin{aligned}
    \rho^2\big|Q(w_1)-Q(w_2)\big|&\leq\frac{1}{4}n(n+2)\int^1_0 \rho^2\Big|(\hat{v}+\tau w_1)^{\frac{4}{n-2}}-(\hat{v}+\tau w_2)^{\frac{4}{n-2}}\Big|\,d\tau 
\\
    &=\frac{1}{4}n(n+2)\int^1_0 \Big| f\Big(\frac{\hat{v}+\tau w_1}{\xi}\Big)-f\Big(\frac{\hat{v}+\tau w_2}{\xi}\Big)\Big|\,d\tau
\\
    &\leq \frac{1}{4}n(n+2)\|f'\|_{L^\infty([1-\delta,1+\delta])}\int^1_0 \tau\big|\frac{w_1-w_2}{\xi}\big| \,d\tau
\\
    &\leq C\rho^{s-2}|w_1-w_2|.
\end{aligned}
\end{equation*}
Arguing similarly as above, we get
\begin{equation*}
\begin{aligned}
    \Big|\nabla(Q(w_1)-Q(w_2))\Big|
    &\leq \frac{1}{4}n(n+2)\int_0^1
        \Big|\nabla f(\hat{v}+\tau w_1)-\nabla f(\hat{v}+\tau w_2)\Big|\, d\tau
\\
    &=\int_0^1 \Big| f'(\hat{v}+\tau w_1)\Big|\tau\big[\nabla w_1-\nabla w_2]
    \\
    &\quadd\quadd + \big[f'(\hat{v}+\tau w_1)-f'(\hat{v}+\tau w_2)\big]\big(\nabla\hat{v}+\tau\nabla w_2\big) \, d\tau
\\
    &\leq \int_0^1  \big| f'(\hat{v}+\tau w_1)\big|\big|\nabla w_1-\nabla w_2 \big|\, d\tau 
    \\
    &\quadd\quadd+\int_0^1 \big|f'(\hat{v}+\tau w_1)-f'(\hat{v}+\tau w_2)\big|\big|\nabla\hat{v}+\tau\nabla w_2\big| \, d\tau.
\end{aligned}
\end{equation*}
Estimating the first term above, we have
\begin{equation*}
    \begin{aligned}
        \rho^3 \int_0^1  \big| f'(\hat{v}+\tau w_1)\big|\big|\nabla w_1-\nabla w_2 \big|\, d\tau 
        &=\rho^{s-1}\int_0^1  \big| f'\Big(\frac{\hat{v}+\tau w_1}{\xi}\Big)\big|\big|\nabla w_1-\nabla w_2 \big|\, d\tau
        \\
        &\leq \rho^{s-1}\|f'\|_{L^\infty([1-\delta,1+\delta])}\big|\nabla w_1-\nabla w_2 \big|
        \\
        &\leq C\rho^{s-1}\big|\nabla\big(w_1-w_2\big)\big|,
    \end{aligned}
\end{equation*}
    and 
\begin{equation*}
\begin{aligned}
&\rho^3\int_0^1 \big|f'(\hat{v}+\tau w_1)-f'(\hat{v}+\tau w_2)\big|\big|\nabla\hat{v}+\tau\nabla w_2\big| \, d\tau
\\
&\quadd=\rho^{s-1}\int_0^1 \big|f'\Big(\frac{\hat{v}+\tau w_1}{\xi}\Big)-f'\Big(\frac{\hat{v}+\tau w_2}{\xi}\Big) \big|\big|\nabla\hat{v}+\tau\nabla w_2\big| \, d\tau
\\
&\quadd\leq \rho^{s-1}\|f''\|_{L^\infty([1-\delta,1+\delta])} \int_0^1 \tau\big|\frac{w_1-w_2}{\xi} \big| \big|\nabla\hat{v}+\tau\nabla w_2\big| \, d\tau
\\
&\quadd= \rho^{2s-3}\|f''\|_{L^\infty([1-\delta,1+\delta])}\big|w_1-w_2 \big| \int_0^1 \tau \big|\nabla\hat{v}+\tau\nabla w_2\big| \, d\tau
\\
&\quadd\leq C\rho^{3s-4}\big|w_1-w_2\big|.
\end{aligned}
\end{equation*}
Therefore,
\begin{equation*}
    \rho^3\big|\nabla(Q(w_1)-Q(w_2))\big| \leq C\{ \rho^{3s-4}\big|w_1-w_2\big| +\rho^{s-1}\big|\nabla\big(w_1-w_2\big)\big|\}.
\end{equation*}
Combining the estimates above, we have
\begin{equation*}
\begin{aligned}
    &\rho^2\big|Q(w_1)-Q(w_2)\big|+\rho^3\big|\nabla(Q(w_1)-Q(w_2))\big| 
        \\
    &\quad\leq C\{\rho^{s-2}|w_1-w_2|+\rho^{3s-4}\big|w_1-w_2\big| +\rho^{s-1}\big|\nabla\big(w_1-w_2\big)\big|\}.
        \\
    &\quad=C\rho^{s-2}\{  |w_1-w_2|+\rho^{2s-2}\big|w_1-w_2\big| +\rho\big|\nabla\big(w_1-w_2\big)\big|\}
        \\
    &\quad\leq C\rho^{s-2}\{  |w_1-w_2|+\rho\big|\nabla\big(w_1-w_2\big)\big|\}.
\end{aligned}
\end{equation*}
In summary,
\begin{equation*}
        \rho^2\big|Q(w_1)-Q(w_2)\big|+\rho^3\big|\nabla(Q(w_1)-Q(w_2))\big| \leq C\rho^{s-2}\{  |w_1-w_2|+\rho\big|\nabla\big(w_1-w_2\big)\big|\}.
\end{equation*}
Multiplying both sides by $e^{\mu t}\rho^{-s}$ and using the fact $s>2$, we get
\begin{equation*}
\begin{aligned}
    &e^{\mu t}\rho^{-(s-2)}\{\big|Q(w_1)-Q(w_2)\big|+\rho\big|\nabla(Q(w_1)-Q(w_2))\big|\}
        \\
    &\quadd\leq Ce^{\mu t}\rho^{-s}\rho^{s-2}\{  |w_1-w_2|+\rho\big|\nabla\big(w_1-w_2\big)\big|\}
        \\
    &\quadd\leq Ce^{\mu t}\rho^{-s}\{  |w_1-w_2|+\rho\big|\nabla\big(w_1-w_2\big)\big|\}.
\end{aligned}
\end{equation*}
Thus,
\begin{equation*}
        \|Q(w_1)-Q(w_2)\|_{\Lambda_{\mu,s-2}^{1}([t_0,\infty)\times\Sigma)}\leq C\|w_1-w_2\|_{\Lambda_{\mu,s}^{1}([t_0,\infty)\times\Sigma)}.
\end{equation*}
Hence, by Theorem 2.9, we obtain
\begin{equation*}
\begin{aligned}
    &{\lVert \mathcal{T}(w_1)-\mathcal{T}(w_2)\rVert}_{\Lambda_{\mu,s}^{2,\alpha}([t_0,\infty)\times\Sigma)}
\\
    &\quadd\leq C {\lVert P(w_1)-P(w_2)\rVert}_{\Lambda_{\mu,s-2}^{0,\alpha}([t_0,\infty)\times\Sigma)}
\\
    &\quadd\leq C\{ 
        {\lVert w_1-w_2\rVert}_{\Lambda_{\mu,s}^{1}([t_0,\infty)\times\Sigma)}
        {\lVert Q(w_1)\rVert}_{\Lambda_{0,0}^{1}([t_0,\infty)\times\Sigma)}
\\
    &\quadd\quadd+{\lVert w_2\rVert}_{\Lambda_{0,0}^{1}([t_0,\infty)\times\Sigma)}
        {\lVert Q(w_1)-Q(w_2)\rVert}_{\Lambda_{\mu,s-2}^{1}([t_0,\infty)\times\Sigma)}\}
\\
    &\quadd\leq  CC'(\epsilon(t_0)+Be^{-\mu t_0}){\lVert w_1-w_2\rVert}_{\Lambda_{\mu,s}^{1}([t_0,\infty)\times\Sigma)}+CB e^{-\mu t_0}{\lVert w_1-w_2\rVert}_{\Lambda_{\mu,s}^{1}([t_0,\infty)\times\Sigma)}
\\
    &\quadd\leq C\{C'(\epsilon(t_0)+Be^{-\mu t_0})+Be^{-\mu t_0}\}{\lVert w_1-w_2\rVert}_{\Lambda_{\mu,s}^{2,\alpha}([t_0,\infty)\times\Sigma)}.
\end{aligned}
\end{equation*}
We finish the proof by choosing $t_0$ sufficiently large. 
\par\indent{\it Step 4.} We now finish the proof. By the contraction mapping theorem, we have $w\in \Lambda_{\mu,s}^{2,\alpha}([t_0,\infty)\times\Sigma)$ satisfying $\mathcal{T}(w)=w.$ This yields a solution $w\in \Lambda_{\mu,s}^{2,\alpha}([t_0,\infty)\times\Sigma)$ of $\mathcal{N}(\hat{v}+w)=0.$ Then, $v=\hat{v}+w$ is a solution of $(\star)$.
\end{proof}
We refer to any $\hat{v}$ obeying \eqref{eq:ASS1}-\eqref{eq:ASS2} as an order-$\mu$ approximate solution to \eqref{eq:CYLINDRICALLN} with leading term $\xi$.
\section{Approximate Solutions}\label{sec-APPROX2}
In this section, we present a general construction of approximate solutions via perturbations of solutions to the linearized problem. For any prescribed order, such approximations can be produced from linearized solutions with decay at infinity. The coefficients appearing in these linearized modes constitute the free asymptotic data: distinct choices lead to distinct approximate solutions at that prescribed order. For the proof of the following proposition, we require a technical lemma due to \cite{HJS2024}; for convenience, we state it here.
\begin{lemma}\label{lem:wjdecay}
Let $\Sigma\subsetneq \mathbb{S}^{n-1}$ be a domain with $\partial\Sigma\in C^{1,\alpha}$ for some $\alpha\in(0,1)$, $s$ the positive constant given by \eqref{eq:s},  $\rho\in C^\infty(\Sigma)\cap \mathrm{Lip}(\Sigma)$ the positive solution of \eqref{eq:rho}-\eqref{eq:rhoB}, and $\mathcal{L}$ the operator in \eqref{eq:ELLv}. Assume $m$ is a nonnegative integer and $\gamma$ is a positive constant.

Suppose $h_0,h_1,\cdots,h_m\in L^2(\Sigma)$ and, for some constants $A,B>0$ and $a>s$,
\begin{equation}
\|h_j\|_{L^2(\Sigma)}\leq B\quad\text{for }j=0,1,\cdots,m,
\end{equation}
and
\begin{equation}
|h_j|\leq A\rho^{a-2}\quad\text{in }\Sigma\text{ for }j=0,1,\cdots,m.
\end{equation}
Then, there exists $w_0,w_1,\cdots,w_m,w_{m+1}\in C^{1,\alpha}(\overline\Sigma)$ such that 
\begin{equation}
\mathcal{L}\bigg(\sum\limits^{m+1}_{j=0} t^je^{-\gamma t}w_j\bigg)=\sum\limits^{m}_{j=0} t^je^{-\gamma t}w_j\quad\text{in }\mathbb{R}\times\Sigma.
\end{equation}
Moreover, for each $j=0,1,\cdots,m+1$,
\begin{equation}
|w_j|+\rho|\nabla_\theta w_j|\leq C(B+A)\rho^s\quad\text{in }\Sigma,
\end{equation}
where C is a positive constant depending only on $n,\alpha,a,\gamma,$ and $\Sigma$.
\end{lemma}
\begin{prop}\label{prop:construct} Let $\xi$ be the positive solution of \eqref{eq:xi}-\eqref{eq:xiB}, $\mathcal{I}$ be the index set associated with $\xi$ (or $\Sigma$), and assume $\mu>\gamma_1$ with $\mu\notin\mathcal{I}$. Suppose that $\eta$ is a solution of $\mathcal{L}\eta=0$ on $\mathbb{R}\times\Sigma$, with $\eta(t,\cdot)\to 0$ as $t\to\infty$ uniformly on $\Sigma$. Then for some $t_0>0$, there exists a smooth function $\tilde{\eta}$ on $[t_0,\infty)\times\Sigma$ such that $\hat{v}=\xi+\eta+\tilde{\eta}$ satisfies \eqref{eq:ASS1} and \eqref{eq:ASS2}. 
\end{prop}
\begin{proof} 
In the proof below, we adopt the notation $f=O(g)$ if $|f|\leq Cg$.
We first decompose the index set $\mathcal{I}$. Set 
    \begin{equation*}
        \mathcal{I}_\gamma=\Bigl\{\gamma_j\,:\, j\geq 1\Bigr\},
    \end{equation*}
    and 
    \begin{equation*}
        \mathcal{I}_{\tilde{\gamma}}=\Bigl\{\sum\limits_{i=1}^k n_i\gamma_i\, :\, n_i\in\mathbb{Z}_+,\,\sum\limits_{i=1}^kn_i\geq2 \Bigr\}.
    \end{equation*}
    We assume $\mathcal{I}_{\tilde{\gamma}}$ is given by a strictly increasing sequence $\{ \tilde{\gamma}_i\}_{i\geq 1}$, with $\tilde{\gamma}_1=2\gamma_1$.
    For convenience, we set
    \begin{equation*}
        \mathcal{N}(v)=v_{tt}+\Delta_\theta v-\frac{1}{4}(n-2)^2v-\frac{1}{4}n(n-2)v^{\frac{n+2}{n-2}},
    \end{equation*}
    and 
    \begin{equation*}
        P(\omega)=\frac{1}{4}n(n-2)[(\xi+\omega)^{\frac{n+2}{n-2}}-\xi^{\frac{n+2}{n-2}}]-\frac{1}{4}n(n+2)\xi^{\frac{4}{n-2}}\omega.
    \end{equation*}
    Take a function $\omega$ such that $|\omega|<\xi$ on $\mathbb{R}\times\Sigma$. We have
    \begin{equation*}
    \begin{aligned}
    \mathcal{N}(\xi+\omega)
    &=\mathcal{L}\omega-P(\omega)
    \\
    &=\mathcal{L}\omega-\frac{1}{4}n(n-2)\Big[(\xi+\omega)^{\frac{n+2}{n-2}}-\xi^{\frac{n+2}{n-2}}\Big]+\frac{1}{4}n(n+2)\xi^{\frac{4}{n-2}}\omega
    \\
    &=\mathcal{L}\omega-\frac{1}{4}n(n-2)\Big[(\xi+\omega)^{\frac{n+2}{n-2}}-\xi^{\frac{n+2}{n-2}}-\frac{n+2}{n-2}\xi^{\frac{4}{n-2}}\omega\Big]
    \\
    &=\mathcal{L}\omega-\frac{1}{4}n(n-2)\xi^{\frac{n+2}{n-2}}\Big[(1+\xi^{-1}\omega)^{\frac{n+2}{n-2}}-1-\frac{n+2}{n-2}\xi^{-1}\omega\Big].
    \end{aligned}
    \end{equation*}
    Consider the Taylor expansion for $|s|<1$,
    \begin{equation*}
        (1+s)^{\frac{n+2}{n-2}}=\sum^{\infty}_{k=0} a_ks^k,
    \end{equation*}
    where each $a_k$ is a constant. Here, we write the infinite sum just for convenience. We do not need the convergence of the infinite series and we always expand up to a finite order. Then, by renaming the constants $\{a_k\}$, we have,
    \begin{equation}\label{eq:powerseries}
        \mathcal{N}(\xi+\omega)=\mathcal{L}\omega+\sum^\infty_{k=2} a_k\xi^{\frac{n+2}{n-2}-k}\omega^k =\mathcal{L}\omega+\rho^{-2-\beta}\sum^\infty_{k=2} a_k\rho^{k\beta}\omega^k.
    \end{equation}
    We point out that the summation above starts from $k=2$.
    
We first consider the case that 
    \begin{equation}\label{eq:EMPTY}
        \mathcal{I}_\gamma \,\cap\mathcal{I}_{\tilde{\gamma}} =\emptyset. 
    \end{equation}
    In other words, no $\gamma_i$ can be written as a linear combination of some $\gamma_1,\cdots,\gamma_{i-1}$ with positive integer coefficients, except a single $\gamma_{i'}$ which is equal to $\gamma_i$. In this case, we arrange $\mathcal{I}$ as follows
    \begin{equation}\label{eq:CHAIN}
        (\beta < )\,\gamma_1\leq \cdots\leq\gamma_{k_1}<\tilde{\gamma}_1<\cdots<\tilde{\gamma}_{l_1}<\gamma_{k_1+1}\leq\cdots\leq\gamma_{k_2}<\tilde{\gamma}_{l_1+1}<\cdots 
    \end{equation}
For each $\tilde{\gamma}_i\in\mathcal{I}_{\tilde{\gamma}} $, we consider nonnegative integers $n_1,\cdots,n_{k_1}$ such that 
    \begin{equation}\label{eq:geq2}
        n_1+\cdots+n_{k_1}\geq 2, \quad n_1\gamma_1+\cdots+n_{k_1}\gamma_{k_1}=\tilde{\gamma}_i. 
    \end{equation}
We proceed in several steps. 

{\it Step 1.} Note that $\gamma_{k_1}<\tilde{\gamma}_1=2\gamma_1$. We take $\eta$ to be a solution of $\mathcal{L}\eta=0$ of the form
    \begin{equation}\label{eq:eta}
        \eta(t,\theta)=\sum^{k_1}_{i=1}c_ie^{-\gamma_it}\phi_i(\theta),
    \end{equation}
where the $c_i$'s are chosen freely. We first take $\omega$ to be $\eta$ as in \eqref{eq:eta}. By \eqref{eq:powerseries} and $\mathcal{L}\eta=0$, we have
\begin{equation*}
\mathcal{N}(\xi+\eta)=
        \sum_{n_1+\cdots+n_{k_1}\geq 2}a_{n_1\cdots n_{k_1}}e^{-(n_1\gamma_1+\cdots+n_{k_1}\gamma_{k_1})t}\phi^{n_1}_{1}\cdots\phi^{n_{k_1}}_{k_1}\rho^{(n_1+\cdots+n_{k_1})\beta},
\end{equation*}
where $n_1,\cdots,n_{k_1}$ are nonnegative integers, and $a_{n_1\cdots n_{k_1}}$ is a constant. By the definition of $\mathcal{I}_{\tilde{\gamma}}$, $n_1\gamma_1+\cdots+n_{k_1}\gamma_{k_1}$ is some $\tilde{\gamma}_i$. Hence, we can write
    \begin{equation}\label{eq:7.38}
        \rho^{-2-\beta}\sum_{i=2}^{\infty}b_i\rho^{i\beta}\eta^i=\sum_{i=1}^{\infty}e^{-\tilde{\gamma}_i t}h_i,
    \end{equation}
    where $h_i$ is given by 
    \begin{equation*}
        h_i=\rho^{-2-\beta}\sum_{(n_1,\cdots,n_{k_1})\in\mathcal{N}_{\tilde{\gamma}_i}} a_{n_1\cdots n_{k_1}}\phi^{n_1}_1\cdots\phi^{n_{k_1}}_{k_1}\rho^{(n_1+\cdots+n_{k_1})\beta}.
    \end{equation*}
    Here, we denote by $\mathcal{N}_{\tilde{\gamma}_i}$ the collection of all $(n_1,\cdots,n_{k_1})$ satisfying \eqref{eq:geq2}.
Therefore, 
\begin{equation}
\mathcal{N}(\xi+\eta)=\sum_{i=1}^{l_1}e^{-\tilde{\gamma}_i t}h_i(\theta)+O(e^{-\tilde{\gamma}_{l_1+1} t}\rho^{s-2}),
\end{equation}
and in particular,
\begin{equation*}
\mathcal{N}(\xi+\eta) = O(e^{-\tilde{\gamma}_1t}\rho^{s-2}).
\end{equation*}

{\it Step 2.}
We claim there exists a function $\tilde{\eta}_1$ such that
\begin{equation*}
\mathcal{N}(\xi+\eta+\tilde{\eta}_{1}) = O(e^{-\tilde{\gamma}_{l_1+1}t}\rho^{s-2}).
\end{equation*}
Note that $3\gamma_1 \in \mathcal{I}_{\tilde{\gamma}}$. We proceed in several cases.

{\it Case 1.} Assume $\gamma_{k_1+1}<3\gamma_1$. Then $\tilde{\gamma}_{l_1}<\gamma_{k_1+1}<3\gamma_1$ and $\tilde{\gamma}_{l_1+1}\leq3\gamma_1$.
Denote 
\begin{equation}\label{eq:I1}
I_1=\sum_{i=1}^{l_1}e^{-\tilde{\gamma}_it}h_i(\theta).
\end{equation}
Then,
\begin{equation}
\mathcal{N}(\xi+\eta)=I_1+O(e^{-\tilde{\gamma}_{l_1+1} t}\rho^{s-2}).
\end{equation}
We need to solve
\begin{equation}\label{eq:previous equation}
\mathcal{L}\tilde{\eta}_1=-I_1.
\end{equation}
Note that $\gamma_m\neq \tilde{\gamma}_i$ for any $i$ and $m$. By Lemma \ref{lem:wjdecay} with $m=0$ and $\gamma=\tilde{\gamma}_i$ for $i=1,\cdots,l_1$, \eqref{eq:previous equation} admits a solution $\tilde{\eta}_{1}$ of the form
\begin{equation}\label{eq:eta1tilde}
\tilde{\eta}_{1}=\sum\limits^{l_1}_{i=1}e^{-\tilde{\gamma}_it}w_i(\theta).
\end{equation}
In other words, we solve
\begin{equation}
\mathcal{L}(e^{-\tilde{\gamma}_it}w_i(\theta))=-e^{-\tilde{\gamma}_it}h_i(\theta).
\end{equation}
Then,
\begin{equation}
\mathcal{N}(\xi+\eta+\tilde{\eta}_1)-\mathcal{N}(\xi+\eta)=\mathcal{L}(\tilde{\eta}_1)+(P(\eta)-P(\eta+\tilde{\eta}_1)),
\end{equation}
so that 
\begin{equation}
\mathcal{N}(\xi+\eta+\tilde{\eta}_1)=(P(\eta)-P(\eta+\tilde{\eta}_1))+O(e^{-\tilde{\gamma}_{l_1+1} t}\rho^{s-2}).
\end{equation}
Now, we expand
\begin{equation}
P(\eta)-P(\eta+\tilde{\eta}_1)=\rho^{-2-\beta}\sum\limits_{m\geq 2}\sum\limits_{\ell=1}^{m}c_{m,\ell}\rho^{m\beta}\eta^{m-\ell}\tilde{\eta}_1^{\ell},
\end{equation}
and since
\begin{equation}
|\eta\tilde{\eta}_1|\leq Ce^{-3\gamma_1t}\rho^{2s},\quad \tilde{\eta}_1^2\leq Ce^{-4\gamma_1t}\rho^{2s},
\end{equation}
we have 
\begin{equation}
P(\eta)-P(\eta+\tilde{\eta}_1)=O(e^{-3\gamma_1t}\rho^{\beta-2+2s}).
\end{equation}
Hence,
\begin{equation*}
\mathcal{N}(\xi+\eta+\tilde{\eta}_{1}) = O(e^{-\tilde{\gamma}_{l_1+1}t}\rho^{s-2}).
\end{equation*}

{\it Case 2.} We now assume $\gamma_{k_1+1}>3\gamma_1$. Then, $\tilde{\gamma}_{l_1}\geq3\gamma_1$. Let $n_1$ be the largest integer such that $\tilde{\gamma}_{n_1}<3\gamma_1$. Then, $\tilde{\gamma}_{n_1+1}=3\gamma_1$. We repeat the same argument as in Case 1 with $n_1$ replacing $l_1$. In the definition of $I_1$ in \eqref{eq:I1}, we in turn take the summation from $i=1$ to $n_1$ and define
\begin{equation*}
\tilde{\eta}_{11}(t,\theta)=\sum\limits^{n_1}_{i=1}e^{-\tilde{\gamma}_it}w_i(\theta),
\end{equation*}
for appropriate $w_i$. We note that there is no $\gamma_i$ between $\tilde{\gamma}_1$ and $\tilde{\gamma}_{n_1+1}$.
Hence, arguing in the same way as Case 1, we get
\begin{equation*}
\mathcal{N}(\xi+\eta+\tilde{\eta}_{11}) = O(e^{-\tilde{\gamma}_{n_1+1}t}\rho^{s-2})=O(e^{-3\gamma_1t}\rho^{s-2}).
\end{equation*}
We are now in a similar situation as the beginning of Step 2, with $\tilde{\gamma}_{n_1+1}=3\gamma_1$ replacing $\tilde{\gamma}_1=2\gamma_1$. If $\gamma_{k_1+1}<4\gamma_1$, we proceed as in Case 1. If $\gamma_{k_1+1}>4\gamma_1$, we proceed as in the beginning of Case 2 by taking the largest $n_2$ such that $\tilde{\gamma}_{n_2}<4\gamma_1$. After finitely many steps, we reach $\tilde{\gamma}_{l_1}$. This completes Step 2.

{\it Step 3.} We are in a similar situation as Step 1, with $\tilde{\gamma}_{l_1+1}$ replacing $\tilde{\gamma}_{1}$. Repeating the argument with $k_1+1$, $k_2$, and $l_1+1$ replacing $1$, $k_1$, and $1$ respectively, and noting $\gamma_{k_2}<\tilde{\gamma}_{l_1+1}$, set 
\begin{equation*}
    \eta_2(t,\theta)=\sum^{k_2}_{i=k_1+1}c_ie^{-\gamma_it}\phi_i(\theta).
\end{equation*}
Arguing as in Step 1, we have
\begin{equation*}
\mathcal{N}(\xi+\eta+\tilde{\eta}_{11}+\eta_2) = \sum_{i=l_1+1}^{l_2}e^{-\tilde{\gamma}_it}h_i(\theta)+O(e^{-\tilde{\gamma}_{l_2+1} t}\rho^{s-2}),
\end{equation*}
for appropriately chosen functions $h_i$.

{\it Step 4.} The situation is again similar to Step 2. Denote 
\begin{equation*}
\mathcal{I}_2=\sum_{i=l_1+1}^{l_2}e^{-\tilde{\gamma}_it}h_i(\theta),
\end{equation*}
and
\begin{equation}
\tilde{\eta}_{2}(t,\theta)=\sum\limits^{l_2}_{i=l_1+1}e^{-\tilde{\gamma}_it}w_i(\theta).
\end{equation}
As in to Step 2, we solve $\mathcal{L}\tilde{\eta}_2=-I_2$ and examine
\begin{equation}
P(\eta+\tilde{\eta}_{11}+\eta_2)-P(\eta+\tilde{\eta}_{11}+\eta_2+\tilde{\eta}_{2}).
\end{equation}
The associated expansion contains at least one factor from $\eta+\tilde{\eta}_{11}+\eta_2$ and at least 1 factor from $\tilde{\eta}_{2}$. By assumption, the exponential decay is at least $\tilde{\gamma}_{l_1+1}+\gamma_1\geq\tilde{\gamma}_{l_2}$. Hence,
\begin{equation*}
\mathcal{N}(\xi+\eta+\tilde{\eta}_{11}+\eta_2+\tilde{\eta}_{2}) =O(e^{-\tilde{\gamma}_{l_2+1} t}\rho^{s-2}).
\end{equation*}

Next, we consider the general case; namely, some $\gamma_i$ can be written as a linear combination of some $\gamma_1,\cdots,\gamma_{i-1}$ with positive integer coefficients. We will modify the discussion above to treat the general case. Whenever some $\gamma_i$ coincides with some $\tilde{\gamma}_{i'}$, an extra power of $t$ appears when solving the linear equation $\mathcal{L}w=f$ according to Lemma \ref{lem:wjdecay}, and such a power of $t$ will generate more powers of $t$ upon iteration. 

As an illustration, consider $\gamma_{k_1}=\tilde{\gamma}_1$ as opposed to strict inequality in \eqref{eq:CHAIN}. This is the first time that some $\gamma_i$ may coincide with some $\tilde{\gamma}_{i'}$. We proceed similarly as in Step 1. Take $k_*\in\{1,\cdots,k_1-1\}$ such that 
\begin{equation*}
\gamma_{k_*}<\gamma_{k_*+1}=\cdots=\gamma_{k_1}=\tilde{\gamma}_1=2\gamma_1.
\end{equation*}
Following the argument of Step 1 and the beginning of Step 2, while solving \begin{equation}
\mathcal{L}\tilde{\eta}_1=-I_1,
\end{equation}
we follow the exact same procedure and obtain, for $i=2,\cdots,l_1$, 
\begin{equation}
\mathcal{L}(e^{-\tilde{\gamma}_it}w_i(\theta))=-e^{-\tilde{\gamma}_it}h_i(\theta).
\end{equation}
Now, for $i=1$, since $\tilde{\gamma}_1=2\gamma_1$, we obtain a solution corresponding to $e^{-\tilde{\gamma}_1 t}$ given by
\begin{equation}\label{eq:above}
te^{-\tilde{\gamma}_1 t}w_{11}(\theta)+e^{-\tilde{\gamma}_1 t}\phi_{k_1}(\theta).
\end{equation}
We then substitute into \eqref{eq:eta1tilde} the new expression given by \eqref{eq:above} for $e^{-\tilde{\gamma}_1t}$. The rest of the proof follows the same argument and can be modified accordingly. We can continue this process until we reach the prescribed decay rate.
\end{proof} 
\bibliographystyle{abbrv}
\bibliography{zhoupaper} 
\end{document}